\documentclass[11pt]{amsart}
\usepackage{amsmath,amsfonts,amsthm,amssymb,mathrsfs,mathtools}
\usepackage{graphicx}
\usepackage[usenames,dvipsnames]{color}
\usepackage{float}
\usepackage{graphicx}
\usepackage{subfigure}
\usepackage{caption}
\usepackage{cite}
\usepackage{url}
\usepackage{caption}
\usepackage{epstopdf}
\usepackage{hyperref}

\usepackage{color}

\newcommand{\bS}{\mathbb{S}} 
\DeclareMathOperator{\sech}{sech}

\newcommand{\Z}{\mathbb{Z}}
\newcommand{\N}{\mathbb{N}}
\newcommand{\R}{\mathbb{R}}
\newcommand{\T}{\mathbb{T}}

\newcommand{\abs}[1]{\left|#1\right|}
\newcommand{\BigOh}[1]{\mathcal{O}(#1)}

\theoremstyle{plain}
\newtheorem{theorem}{Theorem}[section]

\newtheorem{lemma}[theorem]{Lemma}
\newtheorem{proposition}[theorem]{Proposition}

\theoremstyle{definition}
\newtheorem{definition}[theorem]{Definition}

\theoremstyle{remark}

\numberwithin{equation}{section}
\numberwithin{theorem}{section}

\title[Bloch theory for water waves]
   {Bloch theory and spectral gaps \\  for linearized water waves}
\author[W.~Craig]{Walter~Craig}
\address{Department of Mathematics, McMaster University, Hamilton Ontario,
  L8S 4K1 \sc{Canada}, \ craig@math.mcmaster.ca}
\thanks{W.C. is partially supported by the Canada Research Chairs Program and
NSERC through grant number 238452--16.}

\author[M.~Gazeau]{Maxime~Gazeau}
\address{Department of Mathematics, University of Toronto, Toronto Ontario,
  M5S 2E4 \sc{Canada}, \ gazeau@math.toronto.edu}

\author[C.~Lacave]{Christophe~Lacave}
\address{Institut Fourier, CNRS and Universit\'e Grenoble Alpes,
  F-38000 Grenoble \sc{France} \ Christophe.Lacave@univ-grenoble-alpes.fr} 
\thanks{C.L. is partially supported by the Agence Nationale de la
  Recherche through the project DYFICOLTI, grant ANR-13-BS01-0003-01
  and the project IFSMACS, grant ANR-15-CE40-0010.}    

\author[C.~Sulem]{Catherine~Sulem}
\address{Department of Mathematics, University of Toronto, Toronto Ontario,
  M5S 2E4 \sc{Canada},  \ sulem@math.toronto.edu}
\thanks{C.S. is  partially  supported by NSERC through grant
  number 46179--13.}

\begin{document}

\begin{abstract}
The system of equations for water waves, when linearized about
equilibrium of a fluid body with a varying bottom boundary, is
described by a spectral problem for the Dirichlet -- Neumann
operator of the unperturbed free surface. This spectral problem is
fundamental in questions of stability, as well as to the perturbation
theory of evolution of the free surface in such settings. In addition, the
Dirichlet -- Neumann operator is self-adjoint when given an appropriate definition
and domain, and it is a novel but very natural spectral problem for a
nonlocal operator.  In the case in which the bottom boundary varies
periodically, $\{y = -h + b(x)\}$ where $b(x+\gamma) = b(x)$, 
$\gamma \in \Gamma$ a lattice, this spectral problem admits a Bloch decomposition 
in terms of spectral band functions and their associated band-parametrized
eigenfunctions. In this article we describe this analytic construction 
in the case of a spatially periodic bottom variation from constant
depth in two space dimensional water waves problem, giving 
 a description of the Dirichlet -- Neumann operator in terms of the 
 bathymetry $b(x)$ and a construction of the Bloch eigenfunctions
  and eigenvalues as a function of the band parameters. One of the 
consequences of this description is that the spectrum consists of a
series of bands  separated by spectral gaps  which are zones of
forbidden energies. For a given generic periodic bottom profile
$b(x)=\varepsilon \beta(x)$,  every gap opens for a sufficiently
small value of the perturbation parameter~$\varepsilon$.  
\end{abstract}

\maketitle 


\section{Introduction}
\label{Setion_1}


This paper concerns the motion of a free surface of fluid over a
variable bottom, a problem of significance for ocean dynamics in
coastal regions where  waves are strongly affected by the topography.  
There is an extensive  literature devoted to the effect of variable
depth over surface waves and there are many scaling regimes of
interest, including in particular  regimes where the typical
wavelength of surface waves is  assumed to be much longer than the
typical lengthscale of the variations of the bathymetry. For purposes
of many mathematical studies, the variable bottom  topography is
assumed either to be periodic, or else to be described by a
stationary random ergodic process.   

References on the influence of rough bottoms on the free surface
include works of Rosales \& Papanicolaou \cite{RP83},  Craig et al
\cite{CGNS05} \cite{CGS09}, and Nachbin \& S\o lna \cite{NS03}, where
techniques of homogenization theory are used to obtain effective long
wave model equations. The article \cite{CLS12} performs a rigorous
analysis of the effect of a  rapidly varying  periodic bottom in the
shallow water regime. Using  simultaneously  the techniques of
homogenization theory and long-wave analysis, a new model system of
equations is derived, consisting of the classical shallow water
equations that give rise to effective (or homogenized) surface wave dynamics,
coupled with a system of nonlocal evolution equations for a periodic
corrector term. A rigorous justification for this decomposition is
given in \cite{CLS12} in the form of a consistency analysis, in the
sense that the constructed approximated solutions satisfy the water
wave equations up to a small error term that is controlled analytically. 
A central issue in this approach is the question of the  time of
validity of the approximation.  It is shown that the result is valid
for a time interval of duration $\BigOh{1}$ in the shallow water
scaling  {\em only if} the free surface is not in resonance with the rapidly
varying bottom. However resonances are not exceptional. When
resonances occur, secular growth of the corrector terms takes place,
and this compromises the validity of the approximation, and in
particular, a small amplitude, rapidly oscillating bathymetry will
affect the free surface at leading order. The motivation for the
present study is to develop analytical tools that will be useful in
order to address the dynamics of these resonant situations. As a first
step, we consider in this paper the water wave system with a periodic
bottom profile, linearized near the stationary state, and we develop a
Bloch theory for the linearized water wave evolution. This  analysis takes the
form of a spectral problem for the Dirichlet -- Neumann operator of
the fluid domain with periodic bathymetry. 

The starting point of our analysis is the water wave problem written in its 
Hamiltonian formulation. Let 
\[
   \mathcal{S}(b,\eta) = \left\{(x, y): x \in \R, -h +  b(x) < y< \eta(x,t)  \right\}
\]
be the two-dimensional  time-dependent fluid domain where the variable
bottom is  given by $y =-h + b(x)$, and the free surface elevation by
$y = \eta(x,t)$. Following \cite{Z68} and \cite{CS93}, we pose the problem in
canonical variables $(\eta, \xi)$, where $\xi(x) $ is the trace of the
velocity potential on  the free surface $\{y=\eta(x)\}$. In these
variables, the equations of motion for nonlinear free surface water
waves are
\begin{align}\label{WWeq}
	\left\{ \begin{array}{ll}
	\partial_t \eta -G[\eta,   b] \xi =0 , \\
	\partial_t \xi + g\eta +\frac{1}{2}\abs{\partial_x\xi}^2 -
        \dfrac{\left(G[\eta,  b] \xi + \partial_x \eta\cdot \partial_x
          \xi\right)^2}{2(1+\abs{\partial_x\eta}^2)} = 0  ~.
\end{array}\right .
\end{align}
The operator $G[\eta,b] $ is the Dirichlet -- Neumann operator, defined by
\begin{equation}\label{eq2}
	G[\eta,b]\xi =
	\sqrt{1+|\partial_x\eta|^2}\partial_n\varphi_{|_{y = \eta}} ~,
\end{equation}
where $\varphi$ is the solution of the elliptic boundary value problem
\begin{equation}\label{eq2.5}
	\left\lbrace 
	\begin{array}{l}
	\partial_x^2\varphi+\partial_y^2\varphi = 0 \quad 
        \text{in} \quad {\mathcal S}(b,\eta)~,  \\
	\varphi_{\vert_{y = \eta}} = \xi,\qquad
        \partial_n\varphi_{\vert_{y = -h+b }} = 0 ~,
	\end{array}\right. 
\end{equation}
and $g$ is the acceleration due to gravity.
In the present article, we  consider the system of water wave equations,
linearized near a surface at rest and in the presence of periodic bottom.
The bottom  defined as $y = -h + b(x)$ where $b$ is $2\pi$ - 
periodic  in $x$. We assume $b$ is in  $C^1(\T^1)$ where $\T^1$  is
the periodized interval $[0,2\pi)$.  The system \eqref{WWeq} 
linearized about the stationary solution $(\eta(x),\xi(x))=(0,0)$ is as follows. 
\begin{eqnarray}
	\left\{ \begin{array}{ll}
	\partial_t \eta -G[b] \xi =0 , \\
	\partial_t \xi + g\eta =0 ~,
\end{array}\right.
\end{eqnarray}
where now, and for the remainder of this article, we denote $G[0,b]$ by $G[b]$.
This is an analog of the wave equation, however with the usual spatial
Laplacian replaced by the nonlocal operator $G[b]$ whose coefficients
are $2\pi$-periodic dependent upon the horizontal spatial variable $x$:
\begin{equation}\label{wave_eq}
   \partial^2_{t} \eta + g G[b] \eta = 0 ~.
\end{equation}
The initial data for the linearized surface displacement $\eta(x,t)$ are 
\begin{equation} 
   \eta(x,0) = \eta_0(x),  \;\; \partial_t\eta(x,0) = \eta_1(x),  \;\; x \in \R ~,
\end{equation}
these being defined on the whole line.

Bloch decomposition, a spectral decomposition for differential
operators with periodic coefficients,  is a classical tool  to study
wave propagation in periodic media. For a relatively recent example,
Allaire et al.~\cite{APR09} considered the problem of  propagation of
waves  packets through a periodic medium, where the period is assumed
small compared to the size of the envelope of the wave packet.  In
this work the authors construct solutions built upon Bloch plane waves
having a  slowly varying amplitude.   In a study of Bloch
decomposition for the linearized water wave problem over a periodic
bed, Yu and Howard  \cite{YH12} use a conformal map that transforms
the original fluid domain to a uniform strip. Using this map, they
calculate the formal Fourier series for Bloch eigenfunctions. For
various examples of bottom profiles, they compute numerically the
Bloch eigenfunctions and eigenvalues, from which they identify the
spectral gaps and  make several observations of their behavior.  

The main goal of our work in the present paper is to develop Bloch spectral
theory for the Dirichlet -- Neumann operator, in analogy with the classical
case of partial differential operators with periodic coefficients. This theory
constructs the spectrum as a sequence of bands separated by gaps of instability;
it serves as a basis for 
perturbative calculations that gives rise to explicit formulas and rigorous
understanding of spectral gaps, and therefore intervals of unstable
modes of the linearized water wave problem over periodic bathymetry. 

The principle of the Bloch decomposition is to parametrize the
continuous spectrum and the generalized eigenfunctions of the spectral problem
for $G[b] $ on $L^2(\R)$ with a family 
of spectral problems  for $G[b]$ on the interval $[0,2\pi)$, with
$\theta$-periodic boundary conditions. For this purpose,  we construct
the Bloch eigenvalues and eigenfunctions of the spectral problem  
\begin{align}\label{master_eigenvalue_problem}
    G[b]\Phi(x,\theta) = \Lambda(\theta)\Phi(x,\theta) ~, 
\end{align}
with boundary conditions 
\begin{equation} \label{bc}
    \Phi(x+2\pi, \theta) = \Phi(x,\theta) e^{2\pi i\theta} ~,
\end{equation}
where $-1/2 \leq \theta < 1/2$; such behavior is termed to be
$\theta$-periodic in $x$. 
This introduces the band parameter $\theta$.

When the bottom is flat, $b=0$,  the Bloch eigenvalues
$\Lambda^{(0)}_n(\theta)$  are given explicitly in terms 
of the  classical dispersion relation for water waves over a
constant depth, namely   
\begin{equation}  
    \Lambda^{(0)}_n(\theta) = \omega^2(n+\theta) = (n+\theta)\tanh(h(n+\theta))
\end{equation}
for $n \in \N$, and the Bloch parameter $\theta \in \bS^1$,
where $\bS^1$ is the circle $[-1/2,1/2)$ with periodic continuation.
Eigenvalues are simple for $-1/2 < \theta < 0$ and 
$0 < \theta < 1/2$. For half-integer values of $n+\theta$, namely 
$n+ \theta = 0, 1/2, 1 \dots$, eigenvalues $\Lambda^{(0)}_n(\theta)$  
have multiplicity two. If reordered appropriately by their magnitude,
the eigenvalues are continuous and periodic in $\theta$ with period $1$. 
The eigenfunctions $\Phi(x,\theta)$ satisfy the boundary conditions
\eqref{bc}. With the ordering of the eigenvalues specified above, the
eigenfunctions  $\Phi_n^{(0)}(x,\theta)$ are periodic in $\theta$,
again of period $1$.   

Just as in the case of Bloch theory for many second order partial
differential operators, we find that the presence of the bottom
generally results in the splitting of double eigenvalues near such points
of multiplicity, creating a spectral gap. 

\begin{definition}
For $\theta \in \bS^1$ and $b(x)\in C^(\T^1)$ such that $h - b(x)  \ge c_0 >0$, 
the operator $G_\theta[b]$ is defined by
\begin{equation} \label{G-theta}
G_\theta[b] = e^{i\theta x} G[b] e^{-i\theta x}.
\end{equation}
\end{definition}
We will show that   $G_\theta[b]$ maps functions in $H^1(\T^1)$ to
 $L^2(\T^1)$, in particular it takes $2\pi$-periodic functions into $2\pi$-periodic functions.
 We will furthermore show that its spectrum on the domain 
$H^1(\T^1) \subset L^2(\T^1)$ consists of a non-decreasing sequence of eigenvalues 
\[
   \Lambda_0(\theta) \le \Lambda_1(\theta) \le  \cdots \le \Lambda_n(\theta)\le \cdots
\]
which are continuous and periodic in $\theta$. The eigenvalues  are also continuous in $b$,
for $b$ in a $C^1$- neighborhood  $B_R(0)$ of the origin. The
corresponding eigenfunctions $\psi_n(x,\theta)$ are $2\pi$-periodic in
$x$ and periodic in $\theta$ and the corresponding solutions
$\Phi_n(x,\theta)$ of $G[b] \Phi_n(x,\theta) = \Lambda_n(\theta)\Phi(x,\theta)$ 
are $\theta$-periodic. In the case of $(\theta,b)$ such that
$\Lambda_{n-1}(\theta) < \Lambda_n(\theta) < \Lambda_{n+1}(\theta)$,
the eigenvalue $\Lambda_n(\theta)$ is simple, and it and eigenfunction
$\Phi_n(x,\theta)$ are locally analytic in both $\theta$ and $b$.

The spectrum of the Dirichlet -- Neumann operator $G[b]$ on the line,
namely on the domain $H^1(\R) \subset L^2(\R)$, is the union of the
ranges of the Bloch eigenvalues $\Lambda_n(\theta)$, that is
\[
    \sigma_{L^2(\R)} (G[b]) = \cup_{n=0}^{+\infty} [\Lambda_{n}^-,\Lambda_{n}^+]
\]
where $\Lambda_{n}^- = \min_{\theta \in \T^1}\Lambda_{n}(\theta)$ and 
$\Lambda_{n}^+ = \max_{\theta \in \T^1}\Lambda_{n}(\theta)$. It is the
analog of the structure of spectral bands and gaps of the Hill's
operator \cite{McT76}.  The ground state  $\Lambda_0(\theta)$ satisfies $\Lambda_0(0)=0$
for any bathymetry  $b(x)$, and its corresponding eigenfunction is $\Phi(x,0) = 1$.

In Section~\ref{sec-perturbation} we give a perturbation analysis of
spectral behavior and we compute the gap opening for 
$b(x) = \varepsilon \beta(x)$, asymptotically  as a function of~$\varepsilon$.
As an example we consider $b(x) = \varepsilon \cos(x)$, in analogy
with the case of Matthieu's equation, and we calculate the asymptotic
behavior of the first several spectral gaps. We find that, as in
Floquet theory for Hill equation, the first spectral gap obeys 
$| \Lambda_0^+-\Lambda_1^-| = \BigOh{\varepsilon}$. However, in
contrast to the case of the Matthieu equation, the second spectral gap 
only opens at order $| \Lambda_1^+-\Lambda_2^-| = \BigOh{\varepsilon^4}$.  In addition,
we show that the centre of the gap $\frac{1}{2}(\Lambda_1^- + \Lambda_0^+)$ is strictly decreasing in $\varepsilon$.

A generic bottom profile $b(x)$ will open all spectral gaps. Clearly,
the band endpoints $\Lambda_{n}^\pm$ satisfy $\Lambda_{n}^- < \Lambda_{n}^+$ 
unless $\partial_\theta \Lambda_n(\theta)\equiv 0$  
which certainly  does not occur in a perturbative regime.  For
sufficiently small generic bathymetric variations $b(x)$, we also know that 
$\Lambda_{n}^+ \leq \Lambda_{n+1}^-$, which is the case for Hill's
operator, and although we conjecture this to be the case for the
Dirichlet -- Neumann operator for large general $b(x)$, we do not have
a proof of this fact. Furthermore, for Hill's operator, the band edges 
$\{\Lambda_{n}^+,\Lambda_{n+1}^-\}_{n \in \N}$ of the $n^{th}$ gap correspond to
the $4\pi$ periodic spectrum, while we do not have a proof of the
analogous result for the Dirichlet -- Neumann operator.  The reality
condition implies that $\Lambda_n(\theta) = \Lambda_n (-\theta)$, 
and therefore for $n$ even, 
$\partial_\theta \Lambda_{n-1}(0) =\partial_\theta\Lambda_{n}(0)=0$ 
when the $n^{th}$ gap opens. The same holds for $\theta=\pm 1/2$ and $n$ odd.
The existence
of a spectral gap implies that the spectrum is locally simple. Hence
the general theory \cite{R69}\cite{K66} of self-adjoint operators
implies analyticity of both $\Lambda_n(\theta)$ and $\Phi(x,\theta)$. 
In fact, for $\theta\ne 0,\pm1/2$, the unperturbed spectrum is simple
and the same statement of local analyticity holds for $b(x) \subseteq B_R(0)$. 
Gaps are not guaranteed to remain open as the size of the bottom
variations increases, as shown in the numerical simulations 
performed in   \cite{YH12}, Fig.4 (second gap).  

 
\section{The Dirichlet -- Neumann operator}
\label{Section:Dirichlet-Neumann}


The goal is to study the spectral problem
\begin{align}\label{Eqn:master_eigenvalue_problem_1}
    G[b]\Phi(x,\theta) = \Lambda(\theta)\Phi(x,\theta) ~, 
\end{align}
where $G[b]$ is the Dirichlet -- Neumann operator for the fluid domain
$\mathcal{S}(b,0)$. We impose $\theta$-periodic boundary conditions  
\begin{equation}\label{BC0}
    \Phi(x+2\pi, \theta) = \Phi(x,\theta) e^{2\pi i\theta}
\end{equation}
for $\theta \in \bS^1$ the Bloch parameter. It is convenient in Bloch
theory to define 
\begin{equation}
   \psi(x,\theta) = e^{-i\theta x} \Phi(x,\theta)
\end{equation}
to transform the original problem to an eigenvalue problem with
periodic boundary conditions. Indeed, condition~\eqref{BC0}  implies
that $\psi(x,\theta)$ is periodic  in $x$ of period $2\pi$. The spectral 
problem is now rewritten in conjugated form 
\begin{align}\label{Eqn:master_eigenvalue_problem}
   &  G_\theta[b]  \psi(x,\theta) := e^{-i\theta x} G[b] e^{i\theta x} \psi(x,\theta)  
     = \Lambda(\theta) \psi(x,\theta)~.
\end{align}

\subsection{Analysis of the Dirichlet -- Neumann operator}

The following proposition states the basic properties of the Dirichlet
-- Neumann operator with $\theta$-periodic boundary conditions.  

\begin{proposition} \label{prop1}
For each $-1/2 \le \theta <1/2$, the operator $G_\theta[b]$ 
is self-adjoint from $H^1(\T^1)$ to $L^2(\T^1)$ with periodic
boundary conditions. It has an infinite sequence of eigenvalues 
$\Lambda_0(\theta) \le \dots \le \Lambda_n(\theta) \le \Lambda_{n+1}(\theta) \leq \dots $, 
which tend to $\infty$ as $n$ tends to $\infty$ in such a way that
$\Lambda_n(\theta) \sim n/2$.   
\end{proposition}

Writing $D =-i\partial_x$, the Dirichlet -- Neumann operator  $G[b]$ is
written as  
\[
   G[b] = G_0 + DL[b],
\]
where $G_0= D \tanh (hD)$ is the Dirichlet -- Neumann
operator with a flat bottom, and $DL[b]$ is the correction due to the
presence of the topography. In \cite{CGNS05}, it was shown that  
$G[b]= G_0+ DL[b]$ has a  convergent Taylor expansion
in powers of $b$, for $b$ in $B_R(0)$ of $C^1$ and the successive
terms can be calculated explicitly.
Also,
\begin{equation}\label{operatorDL}
   DL[b] = - DB[b] A[b] ~,
\end{equation}
where 
\begin{align}\label{operatorsAB}
  & A[b]  f (x)= \int_{\R} e^{ikx} \sinh(b(x) k) \sech(hk) \hat{f}(k)
  dk \nonumber\\ 
  & B[b] f (x) = \frac{1}{\pi} \int \frac{\partial_{x'} b(x')  (x'-x) +h-b(x')}
    { (x-x')^2 + (b(x')-h)^2} f(x') dx' \\
  & \qquad -\frac{1}{2\pi} \int \ln \bigl( (x-x')^2 + (h-b(x'))^2\bigr) \tilde{G}[-h+b] f(x')
    dx' ~, \nonumber
\end{align}
where $\widetilde{G} [-h+b] $ is the usual Dirichlet -- Neumann operator
in the domain $\mathcal{S}(b,0) = \{-h +b < y < 0\}$ 
that associates Dirichlet data on the boundary $\{y= -h+b(x)\}$ with
Neumann bottom boundary condition at $y=0$, to the normal derivative of the
solution to Laplace's equation on $\{y=-h + b(x)\}$.  Because of the
special decay properties of the integral kernels for $A[b]$ and 
$B[b]$, these operators are well defined on periodic and $\theta$-periodic functions.
In the following, 
\[
G_\theta[b]  = G_\theta[0] + M,
\]
where we use the notation
\begin{align}\label{operators}
   & G_\theta [0]= e^{-i\theta x} G_0 e^{i\theta x}    \\
   & M = e^{-i\theta x} DL[b] e^{i\theta x} = - \bigl(e^{-i\theta x} DB[b]
          e^{i\theta x}\bigr)\bigl(e^{-i\theta x} A[b] e^{i\theta x}\bigr) \nonumber ~.
\end{align}
The operators $M$, $DB_\theta[b] v= e^{-i\theta x} DB[b]
          e^{i\theta x} $ and $A_\theta[b] = e^{-i\theta x} A[b] e^{i\theta x} $
          map $2\pi$-periodic functions to $2\pi$-periodic functions.
The operator $G_\theta[0] $ is unbounded on $L^2(\T^1)$. It is
diagonal in Fourier space variables 
\[
     G_\theta e^{ijx} = 
 (j+\theta) \tanh (h(j+\theta))     e^{ijx} 
     ~.
\]
In the next proposition, we prove that $G[b]$ preserves the class of
$\theta$-periodic functions. 

\begin{proposition}
Given $2\pi$-periodic bottom topography, $b(x + 2\pi) = b(x)$,
suppose that $\xi(x) \in H^1_{loc}(\R)$ is a $\theta$-periodic function
defined on $\mathbb{R}$, namely that 
\begin{equation} \label{xi-theta-per}
     \xi(x + 2\pi) = e^{2\pi i \theta}\xi(x) ~.
\end{equation}
Then the result of application of the Dirichlet -- Neumann operator
$G[b]\xi(x)$ is also  a $\theta$-periodic periodic. That is
\begin{equation} \label{gb-theta-per}
  (G[b]\xi)(x + 2\pi) = e^{2\pi i \theta}G[b]\xi(x) ~.
\end{equation}
\end{proposition}

\begin{proof} 
Let $\varphi(x,y) $ be the harmonic extension of $\xi(x)$ satisfying 
the bottom
boundary conditions $N_{b(x)}\cdot\nabla\varphi(x,y) = 0$. By linearity,
$e^{2\pi i \theta}\varphi(x,y)$ is the harmonic extension of 
$e^{2\pi i \theta} \xi(x)$ satisfying  the same bottom  boundary condition.
On the other hand, the harmonic extension of $ \xi_1(x) :=\xi(x+2\pi)$ is 
$\varphi(x+2\pi,y) $ with the bottom boundary conditions 
$N_{b(x+2\pi )}\cdot\nabla\varphi(x+2\pi,y) =
N_{b(x)}\cdot\nabla\varphi(x+2\pi,y) = 0$, due to the periodicity of $b(x)$.
By uniqueness of solutions,  condition \eqref{xi-theta-per} implies that
$\varphi(x+2\pi,y) =e^{2\pi i \theta}\varphi(x,y)  $, from which 
\eqref{gb-theta-per} follows, namely
\[
   e^{2\pi i \theta}(G[b]\xi)(x) = (G[b]\xi_1)(x) 
   =  G[b](e^{2\pi i\theta}\xi(x)) ~. 
\]
\end{proof}

The next statement shows that the operator $M$ is bounded on $L^2(\T^1)$,
and in fact is strongly smoothing.

\begin{proposition} \label{prop2}
There exists $R>0$ such that for  $b\in B_R(0)$, the ball centered at
the origin and of radius $R$ of 
$C^1(\T^1)$ and $f \in L^2(\T^1)$, $Mf$ is also periodic of period
$2\pi$, and  satisfies the estimate 
\begin{equation}
   \|Mf\|_{L^2}  \ \le \ C_0(|b|_{C^1})  \|f\|_{L^2} ~,
\end{equation}
where the constant $C_0$ 
depends on the $C^1$-norm of $b$,
$C(|b|_{C^1}) = \mathcal{O}(|b|_{C^1})$. 

\noindent
In addition, the operator $M$ is strongly smoothing 
\begin{equation}
   \|Mf\|_{H^s}  \ \le \ C_{rs}(|b|_{C^1})  \|f\|_{H^{-r}} ~.
\end{equation}
for all $r,s>0$.
\end{proposition}
The proof of this proposition is given in Section 5.
As a consequence of these two propositions, the operator $G_\theta$
defined in \eqref{G-theta} maps $H^1(\T^1)$ to $L^2(\T^1)$.

\subsection{Floquet theory}

The spectrum of the Dirichlet -- Neumann operator $G[b]$ acting on the
domain $H^1(\R) \subseteq L^2(\R)$ is real, non-negative, and is
composed of bands and gaps. It is the union over 
$-\frac{1}{2} \le \theta < \frac{1}{2}$ of the Bloch eigenvalues 
$\Lambda_n(\theta)$, the analog to Bloch theory for the Schr\"odinger
operator.  

When $b=0$, the spectrum of $G_\theta$ on $L^2(\T^1)$ consists of the
Bloch eigenvalues $\Lambda^{(0)}_n(\theta)$ 
which are labeled in order of increasing magnitude.
The eigenvalues are
periodic in $\theta$ of period one, and are simple when 
$\theta \ne -\frac{1}{2},0,\frac{1}{2}$. For 
$\theta = -\frac{1}{2},0,\frac{1}{2}$, the spectrum is 
double (see Fig.1).  Denoting $ g_n(\theta) = (n+\theta)\tanh(h(n+\theta))$,  
the eigenvalues and eigenfunctions associated to $G_\theta$ are given
as follows:  
\begin{align*}
  & {\rm For} \ & -\frac{1}{2} \le \theta < 0, \quad  
  & \Lambda_{2n}^{(0)} (\theta) = g_{-n}(\theta);
    \ &\psi_{2n}^{(0)}(x,\theta) = e^{-inx} ~, \\
  & {\rm for} \  & 0  \le \theta < \frac{1}{2}, \quad  
  & \Lambda_{2n}^{(0)} (\theta) = g_{n}(\theta);
    \ &\psi_{2n}^{(0)}(x,\theta) = e^{inx} ~, 
\end{align*}
and
\begin{align*}
  & {\rm for} \ & -\frac{1}{2} \le \theta < 0, \quad  
  & \Lambda_{2n-1}^{(0)} (\theta) = g_{n}(\theta); \ &\psi_{2n-1}^{(0)}(x,\theta) = e^{inx}~, \\
  & {\rm for} \ & 0 \le \theta < \frac{1}{2}, \quad  
  & \Lambda_{2n-1}^{(0)} (\theta) = g_{-n}(\theta);
  \ &\psi_{2n-1}^{(0)}(x,\theta) = e^{-inx} ~.
\end{align*}
With this definition, both $\Lambda_{n}^{(0)}$ and $\psi_{n}^{(0)}$
are periodic in $\theta$ with period $1$  and 
$\Lambda_{n}^{(0)}$ is continuous in $\theta$ while $\psi_{n}^{(0)}$ has 
discontinuities at $\theta = -\frac{1}{2},0,\frac{1}{2}$.

The goal of our analysis is to show that in the presence of a variable
periodic topography, spectral curves which meet when $b = 0$ typically
separate, creating spectrum gaps corresponding to zones of forbidden
energies. For this purpose, assume that the bottom topography is given
by $y(x) = -h + b(x)$ where $b$ is a $2\pi$-periodic function in the
ball $B_R(0) \subseteq C^1$, with $\int_0^{2\pi} b(x) \, dx = 0$,  $h-b(x) \ge c_0>0$. For
our analysis, the circle $\theta \in \bS^1$ of Floquet exponents is
divided into regions in which unperturbed spectra are simple (outer
regions), and regions that include the unperturbed multiple spectra
(inner regions). To apply the method of continuity, these regions are
defined so that they  overlap.   

\begin{theorem}\label{thm1}
For all $\theta \in (-\frac{3}{8},-\frac{1}{8}) \cup
(\frac{1}{8},\frac{3}{8})$, the $L^2$-spectrum of $G_\theta + M$ on
the domain $H^1(\T^1)$ is composed of an increasing sequence of 
eigenvalues $\Lambda_n(\theta) $ that are simple, and analytic in
$\theta$ and $b \in B_R(0)$. The corresponding eigenfunctions
$\Psi_n(x,\theta)$ are normalized $2\pi$-periodic in $x$, and analytic  
in $\theta$ and $b\in B_R(0)$.  
\end{theorem}

The result in Theorem \ref{thm1} is a direct consequence of the
general theory of perturbation of self-adjoint operators 
\cite{R69}. However in Section \ref{sec-perturb}, we provide a
straightforward alternate proof by means of the implicit function
theorem; this approach also serves to motivate the proof of the following result. 

\begin{theorem}\label{thm2}
In the neighbourhood of the crossing points $\theta = 0, \pm\frac{1}{2}$, 
i.e for $\theta \in [-\frac{1}{2},-\frac{5}{16})  \cup
  (-\frac{3}{16},\frac{3}{16}) \cup (\frac{5}{16},\frac{1}{2}]$, 
the spectrum of $G_\theta + M$  on
the domain $H^1(\T^1)$  is composed of an increasing sequence
of eigenvalues $\Lambda_n(\theta)$ which are continuous in $\theta$. 
For $\frac{-3}{16}<\theta <\frac{3}{16}$, the lowest eigenvalue
$\Lambda_0(\theta)$ is simple, and it and the eigenfunction
$\Psi_0(x,\theta)$ are analytic in $\theta$ and $b$. 
\end{theorem}

Both Theorems \ref{thm1} and \ref{thm2} are local in $\theta$. Their
domains of definition overlap on the intervals 
$\theta \in (-\frac{3}{8}, -\frac{5}{16}) \cup 
  (-\frac{1}{8}, -\frac{3}{16}) \cup  
  (\frac{1}{8}, \frac{3}{16}) \cup 
  (\frac{5}{16},  \frac{3}{8})
$.
By uniqueness, in these intervals the eigenvalues and eigenfunctions
agree. Hence they are globally defined periodic functions of $\theta$
in the interval $[-\frac{1}{2}, \frac{1}{2})$. We will focus on
results concerning the opening of spectral gaps at $\theta = 0$ for 
eigenvalues $\Lambda_{2n-1} (\theta) $ and $\Lambda_{2n} (\theta)$;    
this is the topic of Section~\ref{sec-double}. The analysis near the
double points $\theta =\pm \frac{1}{2} $ is similar.

An illustration of  eigenvalues as functions of $\theta$ is given in
Figure \ref{fig1}. The left hand side shows the unperturbed   first
five  eigenvalues in the case of a flat bottom labeled in order of
magnitude.  The right hand side shows these eigenvalues in the
presence of  a small generic bottom perturbation and  the gap
openings. 


\begin{figure}
\centering
  \subfigure{\includegraphics[width=6.25cm]{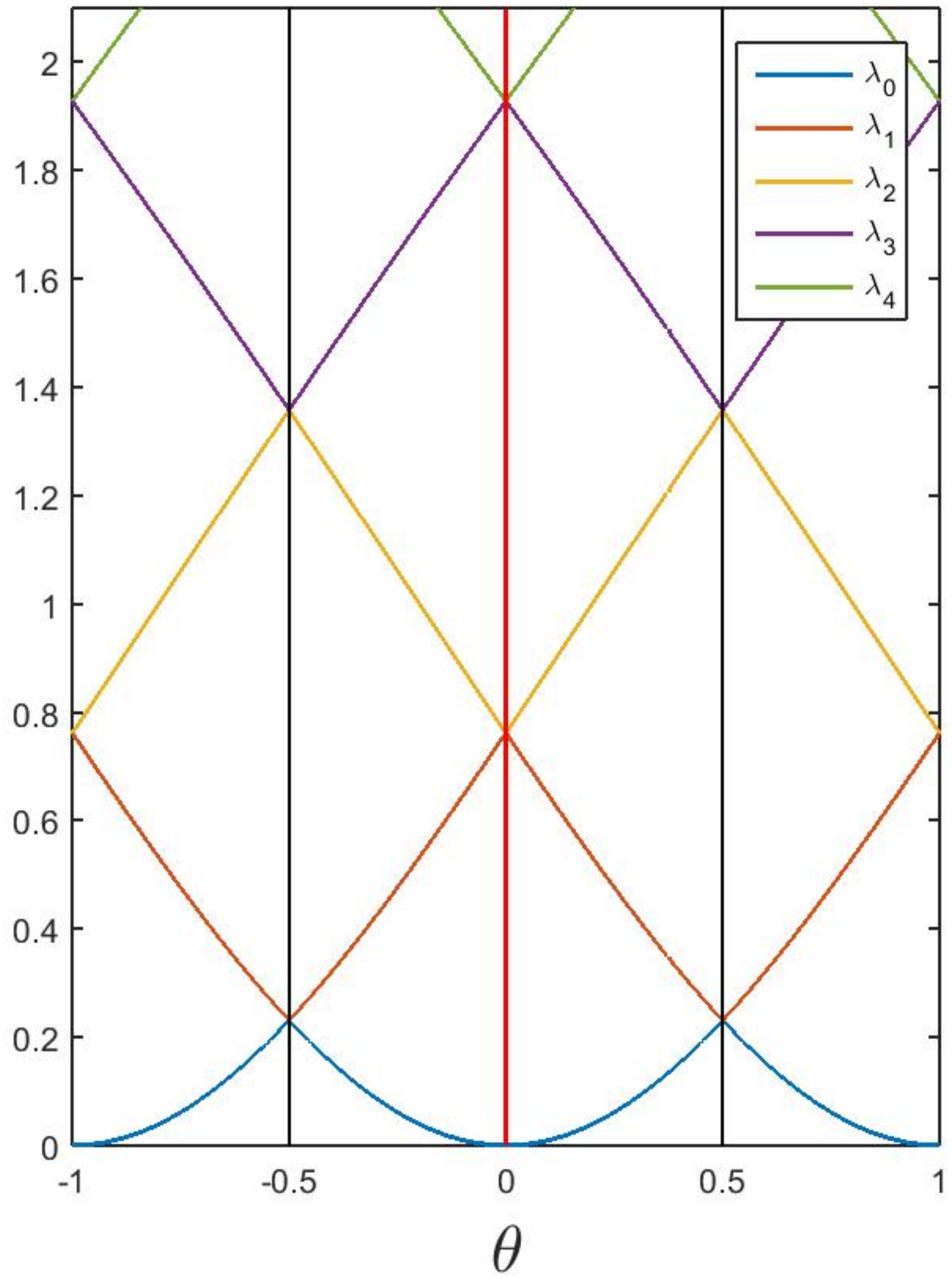}}
  \subfigure{\includegraphics[width=6.25cm]{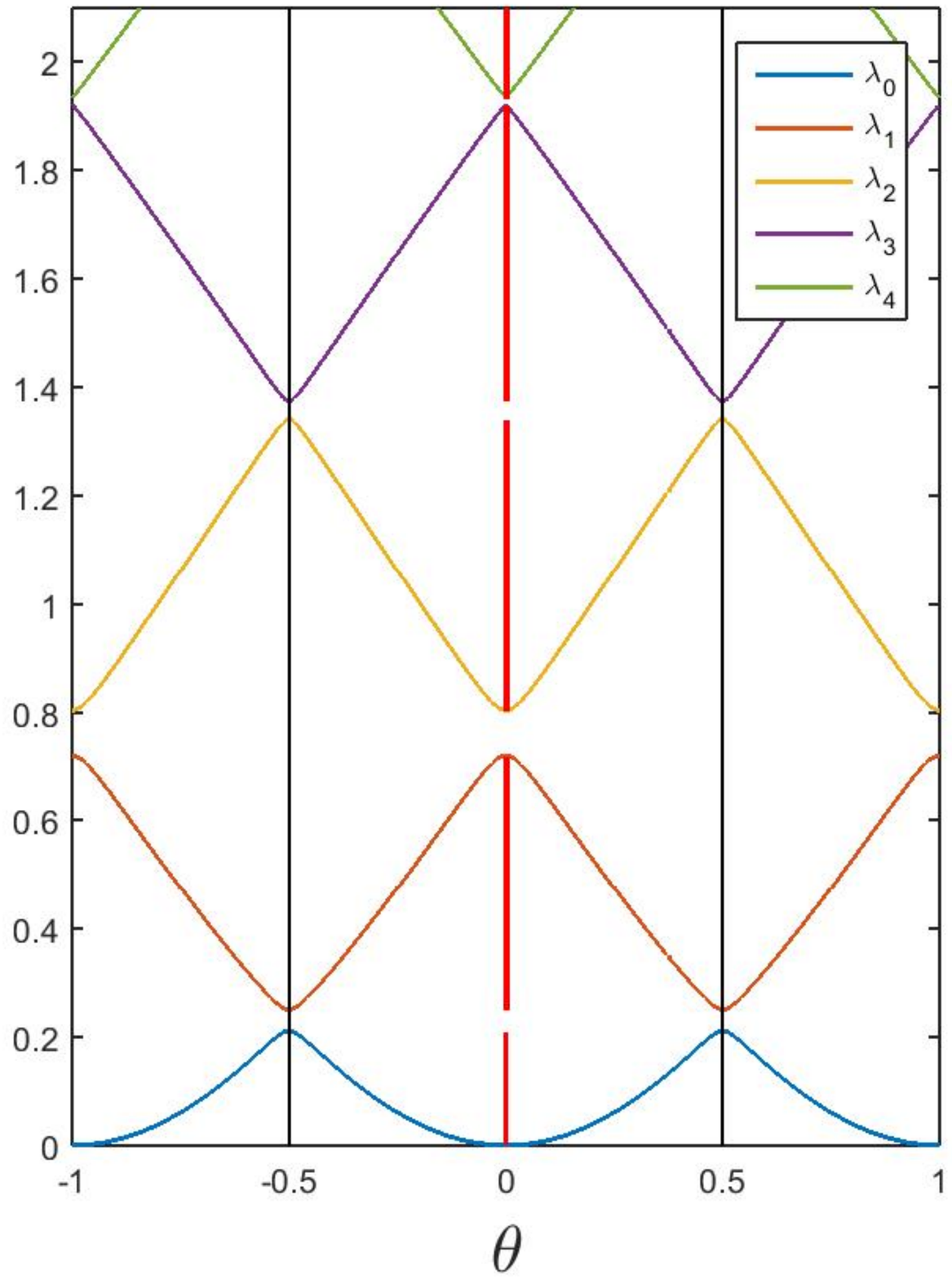}}
  \caption{ First five eigenvalues in order of magnitude:  (left) flat
    bottom; (right) in the presence of a small generic bottom
    perturbation. The vertical axis is the $\Lambda$ axis, and the
    spectra of the operators $G[0]$ and $G[b]$ are represented by the
    vertical red line and the red intervals, respectively.} 
  \label{fig1}
\end{figure}

\section{Gap opening} 
   \label{sec-opening}


\subsection{A finite-dimensional model of gap opening}\label{model}

We describe now, on a simplified model, the mechanism through which
there is the opening of a gap between the two eigenvalues
$\Lambda_{2n-1}$ and $\Lambda_{2n}$ near $\theta=0$ in the presence of
a periodic bathymetric variation $b(x)$. Denote by $P_n$ the orthogonal
projection in $L^2(\T^1)$ onto the subspace spanned by
$\{e^{inx},e^{-inx}\}$ and decompose the operator $G_\theta + M$ as 
\begin{align} \label{decomp}
    G_\theta + M & = P_n (G_\theta+M) P_n + (I-P_n) (G_\theta+M) P_n  \\
   & + P_n (G_\theta+M) (I-P_n) + (I-P_n) (G_\theta+M)(I-P_n) ~.\nonumber
\end{align}
We consider a $2\times 2$ matrix model of \eqref{decomp}, showing that the presence of
a periodic perturbation involving nonzero Fourier coefficients 
$b_{\pm 2n}$ of the bathymetry $b(x)$ leads to a gap between
eigenvalues $\Lambda_{2n-1}(\theta)$ and $\Lambda_{2n}(\theta)$ 
at $\theta = 0$. This model also exhibits how the corresponding
eigenfunctions ${\Psi}_{2n-1} $ and $\Psi_{2n}$ are 
modified. The precise model consists in dropping the three last terms
in the rhs of \eqref{decomp}, reducing $G_0 + M$ to its first term
$P_n( G_\theta+M)P_n$. In addition, we simplify the correction term
$M= e^{-i\theta x} DL[b] e^{i\theta x}$ by replacing $DL[b]$ by its
one term Taylor series approximation in $b(x) \in C^1$, i.e. namely 
\[
   DL_1[b] = -  D \sech(hD) b(x) D \sech(hD) ~.
\]
Acting on Fourier coefficients of a periodic function $\psi(x)$, 
the operator $P_n (G_\theta + e^{-i\theta x}DL_1[b]e^{i\theta x})P_n$ 
is represented by the matrix 
\begin{equation}
   A = \begin{pmatrix} g_n(\theta) & b_{2n} s_n(\theta) s_n (-\theta) \\
       \overline{b}_{2n} s_n(\theta) s_n(-\theta) & g_n(-\theta)
       \end{pmatrix} ~,
\end{equation}
where we have defined $s_n(\theta) = (n+\theta) \sech(h(n+\theta))$. 
We conjugate this matrix to diagonal form as $A = O  \Lambda O^*$
where 
\begin{equation}
   {\Lambda} = \begin{pmatrix} \Lambda_{2n}(\theta) & 0 \\
                 0 & \Lambda_{2n-1}(\theta) \end{pmatrix} ~,
\end{equation}
and where $T= \begin{pmatrix}  0 & \varphi \\ -\varphi & 0 \end{pmatrix}$ 
so that 
$O = \begin{pmatrix} \cos \varphi & \sin \varphi \\ 
     -\sin \varphi &\cos \varphi \end{pmatrix} = e^T= (O^{-1})^*$ 
is a rotation. The eigenvalues 
$\Lambda_{2n-1}(\theta) \leq \Lambda_{2n}(\theta)$ are explicitly given as
\[
   \frac{1}{2} \Big ( g_n(\theta) +g_n(-\theta) \pm \sqrt
        {(g_n(\theta) -g_n(-\theta))^2  + 4 |b_{2n}|^2
          s_n(\theta)^2s_n(-\theta)^2 } \Big) ~. 
\]
Assuming $b_{2n} \ne 0$, the eigenvalues split near $\theta=0$,
and in particular 
$\Lambda_{2n} (0)- \Lambda_{2n-1}(0) = \mathcal{O}(|b_{2n}|)$. 
The corresponding eigenfunctions are 
${ \Psi}_{2n-1}(\theta) = O^* \begin{pmatrix} 1 \\ 0 \end{pmatrix}$  
and ${\Psi}_{2n}(\theta) = O^*\begin{pmatrix} 0 \\ 1 \end{pmatrix}$,
where $\varphi \sim 0$ for $\theta > 0$ and $\varphi \sim \pi/2$ for 
$\theta < 0$. The eigenvalues $\Lambda_{2n-1}(\theta)$,
$\Lambda_{2n}(\theta)$ are Lipschitz continuous in $\theta$. For
nonzero $b_{2n}$, both the matrix $O$ and the eigenvalues $\Lambda_{2n-1}(\theta)$,
$\Lambda_{2n}(\theta)$ are continuous and indeed  even locally analytic. 

\subsection{Perturbation of simple eigenvalues}  
   \label{sec-perturb}

We now take up the spectral problem for the full Dirichlet -- Neumann
operator $G[b]$. Consider values of the parameter $\theta$ in the
interval $\frac{1}{8}  < \theta < \frac{3}{8}$, for which  
$\Lambda_j^{(0)} (\theta) = g_j(\theta)$ is a simple eigenvalue for 
all $j$. With no loss of generality, suppose that $j=2n$ even (for $j$
odd and $\theta > 0$ the only change has to do with the indexing, as
is the case of $\theta < 0$).
By analogy with finite dimensional problem, we seek a conjugacy that
when described in terms of the Fourier transform, will reduce the operator
$G_\theta + M$ to a matrix operator whose off-diagonal entries are zero 
in the $n^{th}$ row and column. Specifically for the $n^{th}$ eigenvalue, we
seek a transformation $O = e^T$ parametrized by operators $T$ satisfying 
$T^*=-T$, such that in acting on Fourier series the matrix
$e^{-T}(G_\theta+M) e^{T} $ will be block diagonal. For this purpose, 
use the orthogonal projection $P_n$ onto the span of the Fourier mode
$e^{inx}$  in $L^2(\T^1)$, and decompose the operator $e^{-T} (G_\theta+M) e^T$ as 
\begin{align}\label{decomp2}
  & e^{-T} ( G_\theta + M) e^T = P_n e^{-T}(G_\theta + M) e^T P_n 
    + (I-P_n)e^{-T} (G_\theta+M) e^T P_n  \\
  & \qquad + P_n e^{-T} (G_\theta+M) e^T (I-P_n) 
    + (I-P_n) e^{-T}(G_\theta+M) e^T(I-P_n) ~. \nonumber
\end{align}
We are seeking $T$ that is an anti-Hermitian operator such that the
block off-diagonal components of \eqref{decomp2} satisfy  
\begin{equation}\label{Eqn:F=0}
  (I-P_n) e^{-T}(G_\theta+M)e^T P_n + P_n e^{-T}(G_\theta+M)e^T (I-P_n) = 0 ~.
\end{equation}
The existence of such $T$ will follow from the implicit function
theorem  \cite{N01}, applied in a space of operators. Define 
$F : \: (T,M) \; \rightarrow \; F(T,M)$, where
\begin{align}
  & F_1(T,M) = P_n  e^{-T} (G _\theta+M)  e^T(I-P_n) ~, \;\; F_2(T,M)
    = F_1^*(T,M) ~, \label{funct-F1-2} \\
  &  F(T,M) = F_1(T,M) + F_2(T,M) ~. \label{funct-F}
\end{align}
The goal is to solve $F(T,M) = 0$, describing $T$ as a function of $M$ in
appropriate functional spaces. We look for $T$ a solution of \eqref{Eqn:F=0} 
restricted to the space of anti-Hermitian operators with the
additional mapping property that    
\begin{equation}\label{antisym}
   T (P_n L^2)  \subseteq (I-P_n) L^2, \quad {\rm{and}} 
       \quad  T ((I-P_n) L^2)  \subseteq P_n L^2 ~.
\end{equation}
The following analysis is performed in Fourier space coordinates,
which is to say in a basis given in terms of Fourier series. 
Denote $h^r$ the space of $2\pi$-periodic functions in
$H^r$, represented in Fourier series coordinates. Alternately using
Plancherel, this characterizes $\psi \in h^r$ by its sequence of
Fourier coefficients; 
\[
     h^r = \{ (\psi_l)_{l\in\Z} \ : \ \sum_{l \in \Z} \langle l \rangle^{2r}|\psi_l|^2
     < +\infty \} ~,
\] 
where as usual we write $\langle l \rangle = \sqrt{2}(|l|^2 + 1)^{1/2}$.

In Fourier space variables, operators have a matrix representation in
the basis $\{\frac{1}{\sqrt{2\pi}}e^{ikx}\}_{k\in \Z}$, defined by
\[
     Ae^{i l x} = \sum_{j\in \Z} A_{jl} e^{ijx} ~. 
\]
For the most part, we will be concerned with the Hermitian and
anti-Hermitian operators.  A scale of operator norms for Hermitian
(and anti-Hermitian) operators represented in Fourier coordinates is
given by    
\[
  \|A\|_r = \sup_j\bigl(\sum_{l\in\Z} |A_{jl}| \langle j - l \rangle^r \bigr) ~.
\]
This norm quantifies the off-diagonal decay of the matrix elements of
$A$. The identical expression for a norm is used for anti-hermitian
operators $T: h^r \to h^r$, while for a general operator $Y$ we need to use
\[
   \| Y \|_r := \Bigl(\sup_j\sum_k |Y_{jk}|\langle j-k \rangle^r \Bigr)^{1/2}
   \Bigl(\sup_k\sum_j |Y_{jk}|\langle j-k \rangle^r \Bigr)^{1/2} .
\]
The space of linear operators from the Sobolev space $h^r$ to itself that
have finite $r$-norm is denoted by ${\mathcal L}_r$. It is a Banach
space with respect to this norm. The space of Hermitian symmetric
operators with finite $r$-norm is denoted by ${\mathcal H}_r$ while
the space of anti-Hermitian symmetric operators with finite $r$-norm
is denoted by ${\mathcal A}_r$. When $r=0$, this norm dominates the
usual operator norm on $\ell^2$, while for $r > 0$ the expression
gives a norm which is a bound for $A \ : \ h^r \to h^r$. 
One notes that this is indeed a proper operator norm, such that 
$\| M \psi \|_{h^r} \leq \| M \|_r \| \psi \|_{h^r}$. In fact
if $\| M \|_r < +\infty$ this same inequality holds when considering
$M: h^s \to h^s$ for any $0 \leq s \leq r$. 

\begin{proposition}\label{Prop:3.1}
Taking $b \in B_R(0) \subseteq C^1(\T^1)$,  then for all $r \geq 0$,
the Fourier representation of the operator $M$ defined in 
\eqref{operators} satisfies 
\[
\|M \|_r =   \sup_j\bigl(\sum_{l\in\Z} |M_{jl}| \langle j - l \rangle^r \bigr)
   \leq C_r  (|b|_{C^1}) = \BigOh{|b|_{C^1}}~.
\]
\end{proposition}

\noindent
This bound follows directly from Proposition~\ref{prop2}. 

We also must consider unbounded operators on sequence spaces, for
instance $A:h^{r} \to h^{r-1}$, for which the operator norm that we
use is given by
\[
   \| A \|_{h^r\to h^{r-1}} := \Bigl(\sup_j\sum_k
      |A_{jk}|\frac{1}{\langle j \rangle}\langle j-k \rangle^r \Bigr)^{1/2}
   \Bigl(\sup_k\sum_j |A_{jk}|\frac{1}{\langle j \rangle}
      \langle j-k \rangle^r \Bigr)^{1/2} .
\]
Denote the space of such operators by 
$\mathcal{H}_{h^r\to h^{r-1}}$.
It is worth the remark that our diagonal operator $G_\theta$ satisfies
$\| G_\theta \|_{h^r\to h^{r-1}} < +\infty$.

\begin{proposition}
For any operators $A:h^{r} \to h^{r-1}$ and $B: h^r \to h^r$ the operator
composition $AB$ satisfies 
\[
    \|AB\|_{h^r\to h^{r-1}} \leq \| A \|_{h^r\to h^{r-1}} \| B \|_r
\]
The same estimate holds for the operator composition $BA$.
\end{proposition}

\noindent
{\textit{Proof:}}  
We note that in our proof, we principally encounter Hermitian and/or
anti-Hermitian operators $A$ and $B$. Their product  does not share
these properties of symmetry however, and we must use the general
expression for the norm for operators  $h^r \to h^r$.  Take $A$ and
$B$ as in the proposition, where we suppose that $\|A\|_{h^r \to h^{r-1}}$ 
and $\|B\|_r$ are finite, and calculate
{\small{
\begin{align*}
   \|AB\|_{h^{r} \to h^{r-1}} & = \Bigl(\sup_j\sum_k\bigl| \sum_l A_{jl}B_{lk} \bigr|
        \frac{1}{\langle j \rangle}\langle j - k \rangle^r \Bigr)^{1/2}
      \Bigl(\sup_k\sum_j \bigl| \sum_l A_{jl}B_{lk} \bigr|
        \frac{1}{\langle j \rangle}\langle j - k \rangle^r \Bigr)^{1/2}\\
   & \leq \Bigl(\sup_j\sum_k\sum_l |A_{jl}|\frac{1}{\langle j \rangle} 
        \langle j - l \rangle^r  |B_{lk}| \langle l - k \rangle^r 
        \Bigl(\frac{\langle j - k \rangle^r}
             {\langle j - l \rangle^r\langle l - k \rangle^r} \Bigr)\Bigr)^{1/2}\\
   & \quad \times \Bigl(\sup_k\sum_j\sum_l |A_{jl}| \frac{1}{\langle j \rangle} 
             \langle j - l \rangle^r  |B_{lk}| \langle l - k \rangle^r
             \Bigl(\frac{\langle j - k \rangle^r}
           {\langle j - l \rangle^r\langle l - k \rangle^r}\Bigr) \Bigr)^{1/2}~.
\end{align*}
}}
Since
\[
  \frac{\langle j - k \rangle^r}
           {\langle j - l \rangle^r\langle l - k \rangle^r} \leq 1     ~,
\]
then the rhs is bounded by $\|A\|_{h^{r} \to h^{r-1}} \|B\|_r$.
\qed

In addition, the family of operators $F$ defined in \eqref{funct-F}
satisfies the mapping properties \eqref{antisym}; these properties
define a linear subspace 
$\mathcal{Q}_{h^r\to h^{r-1} } \subseteq {\mathcal H}_{h^r\to h^{r-1}}$.
The operator $M \in {\mathcal H}_r$ defined in 
\eqref{operators} is a bounded Hermitian symmetric operator, as
expressed in Proposition~\ref{prop1} and Proposition~\ref{Prop:3.1}. 
We seek a solution $T$ of \eqref{Eqn:F=0} in Fourier coordinates, 
where $T \in \mathcal{A}_r$, which also satisfies the mapping
property~\eqref{antisym}. The set of anti-Hermitian operators
satisfying~\eqref{antisym} is a linear subspace 
$\mathcal{P}_r \subseteq \mathcal{A}_r$ 
(it is not however closed under operator composition). 

Describing $T$ in terms of its matrix elements in Fourier variables, 
because of the property~\eqref{antisym}, 
$T e^{inx} = \sum_{l\in\Z\backslash\{n\}} T_{ln} e^{ilx}$, while for 
$j\ne n$ then $T e^{ijx} = T_{nj}e^{inx} = - \overline{T}_{jn}e^{inx}$. Thus the
operator $T$, identified by its matrix elements $T_{nl} = - \overline{T}_{ln}$, 
is nonzero only in the $n^{th}$ row and column, and its
operator norm is given by   
\[ 
    \|T\|_r = \sum_{l\ne n} |T_{nl}|\langle n - l \rangle^r  ~.
\]
Similarly, the operator $M$ is defined in terms of its matrix elements
by its action on Fourier modes, $M e^{ilx} = \sum_{j\in\Z} M_{jl} e^{ijx}$. 
 
\begin{theorem} \label{theor-IFT}
[Simple eigenspace perturbation]. 
There exists $\rho > 0$ and a continuous map  
\begin{equation} \label{F-eq1}
    M \in B_\rho(0) \subseteq {\mathcal H}_{r} \rightarrow  
     T \in {\mathcal P}_{r}
\end{equation}
such that $F(T(M),M) = 0$ ~. 
Furthermore, $T(M)$ is analytic with respect to $M$.
For $b \in B_R(0) \subseteq C^1(\T^1)$, the operator $M=M(b)$ is
continuous with respect to $b$, and $M(0)=0$. 
Therefore, there exists $0<\rho_1<R$ such that for 
$b \in B_{\rho_1}(0) \subseteq C^1(\T^1)$,  $M(b) \in B_{\rho}(0)$
and there is a solution $T=T(b)$ of \eqref{F-eq1}.  Furthermore,  
$M$ is analytic in $b$, therefore $T$ is analytic in $b$.
%
\end{theorem}

The result is that $e^{T(M)}$ is a unitary transformation that
conjugates the operator $(G_\theta + M)$ to diagonal on the eigenspace
associated with the eigenvalues $\Lambda_{2n-1}(\theta)$ and 
$\Lambda_{2n}(\theta)$.  

The proof of  Theorem \ref{theor-IFT} is  an application of the implicit function
theorem, for which we need to verify the hypotheses. As a starting
point, clearly $F(0,0)=0$ since $G_\theta$ is diagonal in a Fourier
basis. We proceed to verify the relevant properties of the mapping $F$
and its Jacobian derivative $\partial_T F$.  This is the object of
Lemma \ref{lemma-IFT}.  

\begin{lemma}\label{lemma-IFT}
Let $r \geq 1$. 
(i) For $(T,M) \in U$ an open subset of   
$ {\mathcal P}_r   \times {\mathcal H}_r$
containing $(0,0)$, $F(T,M) \in \mathcal{Q}_{h^r\to h^{r-1}}$, 
which is continuous in $M$, and $C^1$ in $T$.  

\noindent 
(ii) The Fr\'echet derivative $\partial_T F(0,0)$ at $(T,M) = 0$ 
is invertible, namely for $N \in \mathcal{Q}_{h^r\to h^{r-1}}$,
there is a unique $V \in \mathcal{P}_r$ such that  
\[
   \partial_T F(0,0)V = N ~.
\]
\end{lemma}

\begin{proof}
We will show that $F(T,M)$ has the required smoothness and
invertibility properties as an element of $\mathcal{Q}_{h^r\to h^{r-1}}$, 
in order to invoke the implicit function theorem. Firstly, the
function $F(T,M) \in {\mathcal Q}_{h^r \to h^{r-1}}$ is built of
operators with the following properties: Firstly the operator
$G_\theta$ maps $h^r$ to $h^{r-1}$, and it is diagonal in the Fourier
basis. The operator $e^T \in {\mathcal L}_r$, in fact it is unitary;  
also, $M e^T\in {\mathcal L}_r$. Because 
  $G_\theta \in \mathcal{L}_{h^r\to h^{r-1}}$, then
$G_\theta e^T \in {\mathcal L}_{h^r \to h^{r-1}}.$ We  have 
\[
   e^{-T}(G_\theta + M)e^T \in {\mathcal H}_{h^r \to h^{r-1}} ~,
\]
therefore
\[
     F := P_n[e^{-T}(G_\theta + M)e^T](I-P_n) + (*) \in 
        {\mathcal Q}_{h^r \to h^{r-1}} ~,
\]
where $(*)$ denotes the Hermitian conjugate of the prior
expression. This functional property dictates the choice of sequence
spaces above.   
 
Secondly, we use the series expansion of exponential of operators
\begin{equation}
   e^{T_1} - e^{T_2} = \sum \frac{1}{ m!} ( T_1^m-T_2^m) ~, 
\end{equation}
and furthermore 
\begin{align*}
    T_1^m-T_2^m & = T_1^{m-1} (T_1-T_2) + T_1^{m-2} (T_1-T_2) T_2 
        + .... + (T_1-T_2) T_2^{m-1} \\
      & = \sum_{p=0}^{m-1} T_1^p (T_1-T_2) T_2^{m-p-1} ~.
\end{align*}
Consequently the following operator norm inequality holds: 
\begin{align*} 
  & e^{T_1} - e^{T_2} = \sum_m \frac{1}{ m!} \sum_{p=0}^{m-1} T_1^p
            (T_1 - T_2) T_2^{m-p-1} ~,   \\
  & \|e^{T_1} -e^{T_2}\|_r \le \sum_{m=1}^\infty \frac{1}{(m-1)!}
     a^{m-1} \|T_1-T_2\|_r = e^a \|T_1-T_2\|_r ~,
\end{align*}
where $a = \max (\|T_1\|_r, \|T_2\|_r)$.  

To verify the continuity of $F$ with respect to $T$, compute the difference  
\begin{align}
   F & (T_1,M) - F(T_2,M) \nonumber \\ 
       = & P_n \bigl( e^{-T_1}(G_\theta + M)e^{T_1}
         - e^{-T_2}(G_\theta + M)e^{T_2}\bigr)(I-P_n) + (*) \nonumber \\ 
       = & P_n \bigl( e^{-T_1}G_\theta(e^{T_1} - e^{T_2}) 
        + (e^{-T_1} - e^{-T_2} )G_\theta e^{T_2}\bigr) (I-P_n) \\
    & + P_n \bigl( e^{-T_1} M (e^{T_1} - e^{T_2}) 
        + (e^{-T_1} - e^{-T_2} ) M e^{T_2}\bigr) (I-P_n) + (*) ~,
        \nonumber 
\end{align}
where again $(*)$ is the Hermitian conjugate.
Both $(e^{T_1} - e^{T_2})$ and $e^{T_2}$ are bounded in the operator
$r$-norm and since $r \geq 1$ thus
\begin{align*}
  \| e^{-T_1}G_\theta(e^{T_1} - e^{T_2})\|_{{h^r\to h^{r-1}}}  
    & \leq \|e^{-T_1}G_\theta\|_{{h^r\to h^{r-1}}}e^a
       \|T_1-T_2\|_r   \\
   & \leq \|e^{-T_1}\|_r \|G_\theta\|_{{h^r\to h^{r-1}}}
         e^a\|T_1-T_2\|_r ~,
\end{align*}  
which is the result of Lipschitz continuity. Similarly
\begin{align*}
   \|(e^{-T_1} - e^{-T_2} )G_\theta e^{T_2}\|_{{h^r\to h^{r-1}}}
  & \leq e^a\|T_1-T_2\|_r \|G_\theta\|_{{h^r\to h^{r-1}}}
          \|e^{-T_2}\|_r  ~.
\end{align*} 
Analogous estimates hold for the term involving $M$; 
\begin{align*}
  & \| e^{-T_1} M (e^{T_1} - e^{T_2}) 
        + (e^{-T_1} - e^{-T_2} ) M e^{-T_2}\|_{{h^r\to h^{r-1}}}  \\
    & \qquad \leq e^a\|M\|_r\bigl(\|e^{-T_1}\|_r + \|e^{-T_2}\|_r \bigr)
          \|T_1-T_2\|_r ~.
\end{align*}
The operator $F(T,M)$ is affine linear in $M$, thus the estimate of
continuity is even more straightforward, namely
\begin{align*}
   \|F(T,M_1) - F(T,M_2)\|_{{h^r\to h^{r-1}}} 
     & = \|P_n(e^{-T}(M_1 - M_2)e^{T})(I-P_n) + (*) \|_{h^r\to h^{r-1}}  \\
     & \leq C\|M_1 - M_2\|_r ~.
\end{align*}

We now address the continuity of the Fr\'echet derivative of $F$ with
respect to $T$. For this purpose, we write 
\begin{equation*}
\partial_T F(T,M)  V = P_n (\partial_T  e^{-T}) V (G_\theta+M) e^T  (I-P_n) +
P_n e^{-T} (G_\theta+M) (\partial_T e^{T}) V (I-P_n) +(*)
\end{equation*}
and compute
\begin{equation*}
   (\partial_T e^T) V= \sum_{m=1}^\infty \frac{1}{m!} \sum_{p=0}^{m-1} T^p VT^{m-p-1}.
\end{equation*}
\begin{align}\label{frechet}
   & (\partial_{T_1} e^{T_1}) V - (\partial_{T_2} e^{T_2}) V = 
     \sum_{m=1}^\infty \frac{1}{m!} \sum_{p=0}^{m-1} (T_1^p VT_1^{m-p-1} - T_2^p VT_2^{m-p-1}) \\
   & = \sum_{m=1}^\infty \frac{1}{m!}  \Big(\sum_{p=1}^{m-1} (T_1^p -T_2^p) V T_1^{m-p-1} \Big)
     + \sum_{p=0}^{m-2} T_2^p V (T_1^{m-p-1}- T_2^{m-p-1}) ~. \nonumber 
\end{align}
Let  $a = \max (\|T_1\|_r, \|T_2\|_r)$,  then 
$\|T_1^p-T_2^p\|_r \le p a^{p-1} \|T_1-T_2\|_r$,
\begin{align*}
  & \| \sum_{p=0}^{m-1}   (T_1^p VT_1^{m-p-1} - T_2^p VT_2^{m-p-1})\|_r \le 
  & \sum_{p=0}^{m-1} p a^{p-1}  \|T_1-T_2\|_r  a^{m-p-1}\|V\|_r ~.
\end{align*}
Thus the first term in the rhs of \eqref{frechet} is bounded by 
\begin{align*}
    & \|T_1-T_2\|_r \|V\|_r \sum _m \frac{1}{m!} a^{m-2} \sum_{p=1}^{m-1} p 
    = \|T_1-T_2\|_r \|V\|_r \sum _m \frac{1}{m!} a^{m-2} \frac{m(m-1)}{2}  \\
    & \qquad  =\frac{1}{2} e^a \|T_1-T_2\|_r \|V\|_r ~.
 \end{align*}
The second term of the rhs of \eqref{frechet} has a similar bound. We
combine this estimate with the bounds on $G_\theta$ and $M$ to obtain
the continuity of $\partial_T F(T,M)$ with respect to $T$. 

We now verify that $\partial_T F(0,0)$ is invertible as a mapping 
of linear operators from $\mathcal{P}_r$ to 
$\mathcal{Q}_r$ that satisfy~\eqref{antisym}. 
Since $\partial_T e^T|_{T=0}  V = V$, then
\[
   \partial _T F(0,0) V = P_n [G_\theta,V] (I-P_n) +(*).
\]
Posing the equation 
\begin{equation}\label{Eqn:Commutator}
   \bigl(P_n [G_\theta,V] (I-P_n) +(*) \bigr)_{\pm nl} = N_{ nl} ~, \qquad l\ne  n ~,
\end{equation} 
given $N \in  {\color{blue} \mathcal{Q}_{h^r \to h^{r-1}}  }$ 
we solve for $V \in \mathcal{P}_r$.

We have an explicit inverse
expression for $V$ from the fact that $G_\theta$ is diagonal, namely
\begin{align*}
   V_{jk}  =
   \begin{cases}
   &  \displaystyle{\frac{N_{jn}}{g_\theta(j) - g_\theta( n)} }~,
  \quad j \not=  n ~, \quad k =  n ~, \\
   & \displaystyle{ \frac{N_{ n k}}{g_\theta( n) - g_\theta(k)}} ~,
  \quad j =  n ~, \quad k \not=  n ~,  \\
   &  0   \qquad\qquad\qquad \quad \hbox{\rm otherwise} ~.
   \end{cases}
\end{align*}
The support properties of $V$ follow from this expression, as does the
anti-Hermitian property. Because of eigenvalue separation for $j,k \not=  n$ 
and because of the linear growth property of the dispersion relation $g_\theta(k)$,  
we have a lower bound on the denominator
\[
    |g_\theta(j) - g_\theta( n)| \geq \frac{1}{C} \langle j - n \rangle ~.
\]
The fact that 
$V \in {\mathcal P}_r \subseteq {\mathcal A}_r$ follows from an
estimate of the $r$-operator norm.
\begin{align}
   \| V \|_r & = \sup_j \sum_k |V_{jk}|\langle j - k \rangle^r
   \nonumber \\
      & = \max \Bigl( \sup_{j \not= n} 
          |V_{jn}|\langle j - n \rangle^r ~, ~
          \sum_{k \not=  n}
          |V_{ nk}|\langle n - k \rangle^r\Bigr)  \label{Eqn:V-Est}\\
      & \leq C \max \Bigl( \sup_{j \not=  n} 
          |N_{jn}|\langle j - n \rangle^{r-1} ~, ~
          \sum_{k \not=  n}
          |N_{ nk}|\langle n - k \rangle^{r-1}\Bigr) ~.\nonumber
\end{align}
To finish the argument we observe that, for fixed $n$, the quantity
$\langle j - n \rangle$ has a lower bound
\[
   \langle j - n \rangle \geq C_1(n)  \langle j \rangle ~.  
\]
Therefore, in the expression \eqref{Eqn:V-Est}, we have
\begin{align*}
  \| V \|_r & \leq C \max \Bigl( \sup_{j \not=  n} 
          |N_{jn}|\langle j - n \rangle^{r}
            \frac{1}{\langle j - n \rangle} ~, ~
           \sum_{k \not=  n}
          |N_{ nk}|\langle n - k \rangle^{r}
            \frac{1}{\langle k - n \rangle} \Bigr) \\
    & \leq C \max \Bigl( \sup_{j \not= n} 
          |N_{jn}|\frac{1}{\langle j \rangle}
             \langle j - n \rangle^{r}~,  ~
          \sum_{k \not= n}
          |N_{ nk}|\frac{1}{\langle k \rangle} 
             \langle n - k \rangle^{r} \Bigr)  \\
    & \leq C\|N\|_{{h^r \to h^{r-1}}} ~.
\end{align*}
This proves that 
$(\partial_T F(0,0))^{-1}: {\mathcal Q}_{{h^r \to h^{r-1}}} \to
{\mathcal P}_r$ is bounded.
\end{proof} 
 
Choosing the operator $T$ as in Theorem \ref{theor-IFT}, we now have
the decomposition 
\begin{align}\label{decomp3}
   e^{-T} (G_\theta + M) e^T = & P_n e^{-T} (G_\theta + M) e^{T}P_n \\
  & + (I-P_n) e^{-T} (G_\theta + M) e^{T}(I-P_n) ~ \nonumber
\end{align}
in which the $n$th row and column are both identically zero except
for the coefficient $\Lambda_{2n}(\theta)$ on the diagonal. Therefore 
\[
   e^{-T} (G_\theta + M) e^{T} \delta_{n} 
     = \Lambda_{2n} (\theta)\delta_{n} ~,
\]
where $\delta_n $ is the Kronecker symbol, $\delta_n(j) =0$ $j\ne n, \delta_n(n) =1$.
Since $P_n$ is the projection on the mode $e^{inx}$, an expression
for the eigenfunction in space variables is given by 
$\psi_{2n}(x,\theta) = \sum_{l \in \Z} (e^{T})_{ln}e^{ilx}$, so that
\begin{equation}
   (G_\theta + M)\psi_{2n}(x,\theta) = \Lambda_{2n}\psi_{2n}(x,\theta) ~.
\end{equation}
Furthermore, the transformation $e^{T}$ is unitary, so that
$\|\psi_{2n}(x,\theta)\|_{\ell^2} = 1$. That is, the function
$\psi_{2n}(x,\theta)$ is the normalized eigenfunction of the operator
$(G_\theta + M) $ associated to the eigenvalue $\Lambda_{2n}(\theta)$. 
In addition, since $\|T\|_r < +\infty$, the Fourier series expansion of
the eigenfunction is localized close to the exponential $e^{inx}$, in
that $\delta_n - \hat{\psi}_{2n}(l) \leq C_n\langle l \rangle^{-r}$. 
This is the expression for eigenfunctions and associated eigenvalues
in the case of simple spectrum.  The analysis of $\psi_{2n-1} $ is identical.

\subsection{Perturbation of double eigenvalues} \label{sec-double}

The ground state eigenvalue $\Lambda_0(\theta)$ is simple in the interval
$-\frac{3}{8}<\theta<\frac{3}{8}$ and therefore is also governed by
the general theory of self adjoint operators, as mentioned above, or
alternatively by our construction of the previous section.   

For $1\le n \in \N$, we focus on the $n^{th}$ spectral gap. We perform
a perturbation analysis of spectral subspaces associated with double
or near-double eigenvalues $\Lambda_{2n-1}(\theta)$ and
$\Lambda_{2n}(\theta)$. Denote $P_n$ the projection on the subspace
of $L^2(\T^1)$ spanned by $\{e^{inx}, e^{-inx} \}$ and define the
function $F(T,M)$ as in \eqref{funct-F} and \eqref{funct-F1-2}. 
Analogously, define the space of operators respecting the
tw0-dimensional range of the projection $P_n$ to be ${\mathcal P}_r$  
and ${\mathcal Q}_{h^r \to h^{r-1}}$.
 
\begin{theorem} \label{theor-IFT-2}
 [Perturbation of eigenspaces of double eigenvalues].
 There exists $\rho > 0$ and a continuous map  
\begin{equation} \label{F-eq2}
    M \in B_\rho(0) \subseteq {\mathcal H}_{r} \rightarrow  
     T \in {\mathcal P}_{r}
\end{equation}
such that $F(T(M),M) = 0$ ~. 
Furthermore, $T(M)$ is analytic with respect to $M$.
For $b \in B_R(0) \subseteq C^1(\T^1)$, the operator $M=M(b)$ is
continuous with respect to $b$, and $M(0)=0$. 
Therefore, there exists $0<\rho_1<R$ such that for 
$b \in B_{\rho_1}(0) \subseteq C^1(\T^1)$, $M(b) \in B_{\rho}(0)$ and
there is a solution $T=T(b)$ of \eqref{F-eq2}.  Furthermore, $M$ is
analytic in $b$, therefore $T$ is analytic in $b$. 
\end{theorem}

The proof of this theorem  is similar to the one presented for a
simple eigenvalue and relies on the implicit function theorem. We
present only the steps in the proof of Lemma~\ref{lemma-IFT}   that need
to be modified. 
 
The operator $T$ is such that $T^* = -T$, this is to say 
$T_{\pm n l } = - \overline{T}_{l (\pm n)}$, and the mapping property
\eqref{antisym} of $T$, which is expressed in terms of the  Fourier series
basis elements as the properties that 
$T e^{ \pm inx} = \sum_{l\ne \pm n} T_{ l (\pm n)} e^{ilx}$, and for 
$l\ne \pm n$ then $T  e^{  ilx} = \sum_{\pm n} T_{(\pm n) l}  e^{\pm i n x}$. 
In particular, the nonzero entries of $T_{\pm n l}$ are only for
indices $\pm n,l \ne n$. 
The operator
norm is defined in the usual way as 
\[
    \|T\|_r = \sup_{\pm n} 
      \Bigl(\sum_{l} |T_{(\pm n) l}|\langle l \mp n\rangle^r \Bigr) ~.
\]
When checking the hypotheses for the application of the implicit
function theorem, the only point that is necessary to be verified is
that the operator $\partial_T F(0,0)$ is invertible. To this aim, 
the solution of equation~\eqref{Eqn:Commutator} in this case is explicitly 
\[
    V_{(\pm n)l} =  \frac{N_{(\pm n) l}}{ g_{\pm n}(\theta) - g_l(\theta)}
    ~
\]
for all $ l \ne \pm n$. Since for $-\frac{3}{8}<\theta<\frac{3}{8}$ 
and for $l \ne \pm n$ there is a lower bound on the rhs, namely
\[
     |{g_{\pm n}(\theta) - g_l(\theta)}| > \alpha \langle l \mp n \rangle ~,
\]
for any integer $l\ne \pm n$.
It is again clear that for 
$N \in {\mathcal Q}_{h^r \to h^{r-1}} \subseteq {\mathcal H}_{h^r \to h^{r-1}}$,
 the solution $V$ gives an operator that satisfies 
$V \in {\mathcal P}_r \subseteq {\mathcal A}_r$. 

We note that the eigenvalues themselves are not analytic in $b$, at
least not uniformly in $b$ and in $\theta$ in neighborhoods of double
points. This is exhibited in the model of Section~\ref{sec-opening}. 
But the spectral subspace  spanned by $\{\psi_\pm \}$is analytic, a
common state of affairs in the theory of eigenvalue perturbation in
the case of spectrum with varying multiplicity.  


\section{Perturbation theory of spectral gaps}\label{sec-perturbation}


In this section, we provide a perturbative calculation of the gap
created near $\theta=0$ between the eigenvalues
$\Lambda_{2n-1}(\theta)$ and $\Lambda_{2n}(\theta)$ in the presence of
periodic bathymetry $b(x)$. For clarity of the perturbation
computation, we make the substitution $b(x) = \varepsilon \beta(x)$,
and continue the discussion as a perturbation calculation in $\varepsilon$. 

The first step is to calculate the operator $T$ perturbatively by solving the equation
\[
   P_n e^{-T} (G_\theta+M) e^T (I-P_n)  +(*) = 0 ~. 
\]
Following this, we calculate the eigenvalues $\Lambda_{2n-1}(\theta)$
and $\Lambda_{2n}(\theta)$  of the  $2\times 2$ matrix 
$P_n e^{-T}(G_\theta+M) e^T P_n$ and their corresponding eigenfunctions  
$\Psi_{2n-1}, \Psi_{2n}$. Our calculation provides the eigenfunctions
$\psi_{2n-1} = e^T \Psi_{2n-1}$ and $\psi_{2n} = e^T \Psi_{2n}$ of 
$G_\theta + M$ that are associated to $\Lambda_{2n-1}(\theta)$ and
$\Lambda_{2n}(\theta)$.

\subsection{Expansion of the operator $T$}

Using the Taylor expansion of the Dirichlet -- Neumann operator in powers
of the bottom variations as calculated in \cite{CGNS05}, write
\[
   M = \sum_{p\ge1} \varepsilon^p M_p ~.
\]
We seek the anti-Hermitian operator $T$ in the form 
\[
   T= \sum_{p\ge 1} \varepsilon^p T_p ~.
\]

\begin{proposition}
The coefficients $T_p$ in the series expansion of $T$ are given recursively.
\end{proposition}

\begin{proof}
Writing the exponentials $e^{\pm T}$ in terms of a formal expansion in
$\varepsilon$,  
\[
    e^T = I+\sum_{s\ge1} \varepsilon^s T_+^{(s)}, \qquad 
    e^{-T} = I+\sum_{s\ge1} \varepsilon^s T_-^{(s)} ~,
\]
we have $T_\pm^{(s)}$ computed in terms of $T_p$ as
\begin{equation}
   T_\pm^{(s)} = \sum_{j=1}^s \frac{(\pm  1)^j}{j!} 
       \Bigl( \sum_{p_1+..+p_j=s  \atop 
        1\le p_1, \cdots,p_j \le s-(j-1) } 
       T_{p_1}  \cdots T _{p_j} \Bigr)~.
\end{equation}
In particular
\begin{align}\label{T-s}
   T_\pm^{(s)} &=  \pm T_s + \sum_{j=2}^s \frac{(\pm  1) ^j}{j!} 
      \Bigl( \sum_{p_1+..+p_j=s  \atop 
      1\le p_1,...,p_j \le s-(j-1) } T_{p_1}\cdots T _{p_j}  \Bigr)\\
  & := \pm T_s + h_\pm^{(s)}(T_1,\cdots,T_{s-1})~.  \nonumber
\end{align}
Expand the product 
\begin{align} \label{expansion-full}
  e^{-T}(G_\theta & + M) e^T = (I + \sum_{s\ge 1} \varepsilon^s T_-^{(s)}) 
    (G_\theta +\sum_{p\ge 1} \varepsilon^p M_p) (I+\sum_{l\ge 1} 
    \varepsilon^l  T_+^{(l)}) \nonumber \\
   = & G_\theta + \sum_{q\ge 1} \varepsilon^q  
         ( T_-^{(q)} G_\theta  + G_\theta T_+^{(q)} )
       + \sum_{q \ge 1}    \varepsilon^q  M_q  \\
     & \quad + \sum_{q\ge 2} \varepsilon^q  
         \sum_{q_1+q_3 =q \atop 1\le q_1, q_3 < q} T_-^{(q_1)} G_\theta T_+^{(q_3)}   
       + \sum_{q\ge 2} \varepsilon^q 
         \sum_{ q_1+q_2 + q_3 =q \atop 0\le q_1, q_3 < q; 1\le q_2<q } 
         T_-^{(q_1)} M_{q_2} T_+^{(q_3)}~.  \nonumber 
\end{align}
Apply the projections $P_n$ and $I-P_n$ to the right and left  of
\eqref{expansion-full} respectively,  and set to zero term of order
$q$ ($q\ge1$), 
\begin{align} \label{eq-T}
  & P_n( T_-^{(q)} G_\theta  + G_\theta T_+^{(q)})(I-P_n) = - P_n M_q (I-P_n)  
     \nonumber \\
  & - P_n \Big( \sum_{q_1+q_3 =q \atop 1\le q_1, q_3 < q } 
      T_-^{(q_1)} G_\theta T_+^{(q_3)} 
    + \sum_{ q_1+q_2 + q_3 =q \atop 0\le q_1, q_3 < q; 1\le q_2<q } 
      T_-^{(q_1)} M_{q_2} T_+^{(q_3)} \Big) (I-P_n) ~.
\end{align}
Isolating $T_q$ in the lhs, we get
\begin{align} \label{eq-forT-q}
  & P_n  [G_\theta, T_q ] (I-P_n) =  -  P_n M_q (I-P_n) -
    P_n  ( h_-^{(q)} G_\theta  + G_\theta h_+^{(q)} ) (I-P_n)
    \nonumber \\
  & - P_n \Big( \sum_{q_1+q_3 =q \atop 1\le q_1, q_3 \le q } 
        T_-^{(q_1)} G_\theta T_+^{(q_3)}   + 
        \sum_{ q_1+q_2 + q_3 =q \atop 0\le q_1, q_3 < q, \ 1\le q_2<q } 
       T_-^{(q_1)} M_{q_2} T_+^{(q_3)} \Big) (I-P_n) ~.
\end{align}
The latter equation for $T_q$ has the form 
\[
   P_n  [G_\theta, T_q ] (I-P_n)  = P_n N_q(I-P_n) ~,
\]
where the rhs is defined to be the rhs of \eqref{eq-forT-q}, it depends
on all the previous $T_p$, $1\le p\le q-1$, and its solution is given
explicitly by
\begin{equation}\label{T-q}
   (T_q)_{(\pm n) l} = \frac{(N_q)_{(\pm n) l}}{g_{\pm n} (\theta) -
    g_l(\theta) } ~,  \,  \ell\ne \pm n,
\end{equation}
with the property that $(T_q)_{l(\pm n)} = - (\overline{T_q})_{(\pm n) l}$ 
and $T_{jk}=0 $ otherwise. In particular, for $q=1$, we have
\begin{align}\label{T1-formula}
   P_n  [G_\theta, T_1 ] (I-P_n)  &=  - P_n M_1(I-P_n) \\
  (I-P_n)  [G_\theta, T_1 ] P_n  &=  - (I- P_n) M_1P_n ~. \nonumber
\end{align}
\end{proof}

\begin{lemma} \label{lemma4-2}
The matrix coefficients of the operator $M_p$ are polynomials of order
$p$ in  $\beta_k$, the Fourier coefficients of $\beta$, 
$$(M_p)_{jl} = \sum_{k_j\in \Z} m_p(k_1,....,k_p) \beta_{k_1} .....\beta_{k_p}$$
where $m_p(k_1,...,k_p) $ is nonzero only when $k_1+...+ k_p=j-l$. 

The matrix coefficients of the operator $T_p$ satisfy a similar constraint
$$(T_p)_{jl} = \sum_{k_j\in \Z} t_p(k_1,...,k_p) \beta_{k_1} ...\beta_{k_p}$$
where $t_p(k_1,...,k_p) $ is nonzero only when $k_1+...+ k_p=j-l$. 
\end{lemma}

The proof of this lemma is by induction, first using the recursion
formula for the operator $M$, as described in
\cite[Appendix]{CGNS05}, and  then using formula 
\eqref{T-s} and \eqref{T-q} for $T$.

We conclude this section by stating the first three terms of the
expansion of $T$. For $\ell \ne \pm n$,  
\begin{equation}\label{Tq-G}
\begin{aligned}
  & [T_1,G_\theta]_{\pm n \ell}  = (M_1) _{ \pm n \ell} \\   
  &[T_2,G_\theta]_{\pm n \ell} =  (M_2 + [M_1,T_1])_{\pm n \ell} \\
  &[T_3,G_\theta]_{\pm n \ell}  = (M_3 + [M_2,T_1] + [M_1,T_2] +
  \frac{1}{3} [T_1,[T_1,M_1]] )_{\pm n \ell}~. 
\end{aligned}
\end{equation}

\subsection{Eigenvalue expansion}
Define the $2 \times 2$ matrix coefficient 
\begin{equation} \label{matrix-A}
A(\theta, \varepsilon) = P_n e^{-T}(G_\theta +M) e^T P_n.
\end{equation}
We  compute the  two eigenvalues $\Lambda_{2n-1}(\theta) $ and
$\Lambda_{2n}(\theta)$  perturbatively, and thus give conditions for
the opening of the $n^{th}$ spectral gap. The order at which  the gap
opens depends on the Fourier coefficients of $\beta(x)$. Using the
expansion of $T$, we find 
\begin{align}\label{pn-exp-pn}
  P_n  e^{-T}(G_\theta+M &)e^{T} P_n = P_nG_\theta P_n+ \varepsilon P_n M_1P_n  \nonumber \\
  + \varepsilon^2 P_n \Big(&  \frac{1}{2} [T_1,[T_1,G_\theta]] + M_2+ [M_1,T_1] \Big) P_n \\
  + \varepsilon^3 P_n \Big(& \frac{1}{2} [T_1,[T_2,G_\theta]]  +
  \frac{1}{2} [T_2,[T_1,G_\theta]]  + M_{3}  \nonumber \\ 
 &   + [M_2,T_1] + [M_1,T_2] + \frac{1}{2} [T_1,[T_1, M_1]]  \Big) P_n
  + \BigOh{\varepsilon^4}~. \nonumber 
\end{align}
From the Taylor expansion of $G[b]$ in powers of $b$ (see Appendix of \cite{CGNS05}),
we have
\begin{align*}  
   M_1 & = -e^{-i\theta x} D\sech(hD) \, \beta \, D\sech(hD) e^{i\theta x} \\
   M_2 & = -e^{-i\theta x} D\sech (hD) \, \beta \, D\tanh(hD) \,
       \beta(x) \, D \sech(hD) e^{i\theta x}  \nonumber \\
   M_3 & = e^{-i\theta x} D\sech(hD) \ \Big( -\frac{1}{6} \beta^3 D^2 +
   \frac{1}{2} \beta^2 D \beta \nonumber \\ 
 & \qquad    -\beta D \tanh(hD) 
   \beta  D\tanh(hD) \beta\Big)
   \, D \sech(hD) e^{i\theta x}  \nonumber \\
 M_4 & = -e^{-i\theta x} D\sech(hD) \Big( \frac{1}{2} \beta^2 D^2 \beta D \tanh(hD) \beta 
   -\frac{1}{6} \beta  D\tanh(hD) \beta^3 D^2 \nonumber\\
 & \qquad  +\frac{1}{2} \beta D \tanh(hD) \beta^2 D^2 \beta
   -\frac{1}{6} \beta^3 D^3 \tanh(hD) \beta \nonumber \\ 
 & \qquad -\beta  D\tanh(hD)  \beta  D\tanh(hD)  \beta  D\tanh(hD) \Big)
   \, D \sech(hD) e^{i\theta x} ~. 
\end{align*}
Recalling the notations $ g_k\theta) = ( k+\theta)  \tanh (h(k+\theta))$ and 
$ s_k(\theta) = (k+\theta)  \sech(h(k+\theta))$, we have 
\begin{equation}
A_1 :=    P_n M_1 P_n = \begin{pmatrix}
      0 & - \beta_{2n} s_n(\theta) s_n (-\theta) \\
     -\overline{\beta}_{2n} s_n(\theta) s_n(-\theta) & 0
\end{pmatrix},
\end{equation}
where $A=\sum_{p\ge 0} \varepsilon^p A_p$. Under the condition that 
$\beta_{2n} \ne 0$, the $n{^{th}} $ gap occurs for $\theta = 0$, and
the size of the gap is of order $\mathcal{O}(\varepsilon)$,
approximated by the difference of the two eigenvalues of the above
matrix. This recovers the prediction given by the $2\times 2$ model of
Section \ref{model}. Indeed, computing $A_1$ close to the value 
$\theta = 0$, the spectral gap is approximately given by 
\begin{equation} \label{app-diff}
   \Lambda_{2n}(0)-\Lambda_{2n-1}(0) \sim 2\varepsilon s_n^2(0) |\beta_{2n}|~.
\end{equation}

If $\beta_{2n}=0$, we may pass to the next term in the expansion.
Returning to eq. \eqref{pn-exp-pn}, the term of order
$\mathcal{O}(\varepsilon^2)$ reduces to
\begin{equation*}
   P_n\Big( \frac{1}{2} [T_1,[T_1,G_\theta]] +
     M_2+ [M_1,T_1] \Big) P_n.
\end{equation*}
The eigenvalues $\Lambda_{2n-1} (\theta)$ and $\Lambda_{2n}(\theta)$
are approximated at order 
$\mathcal{O}(\varepsilon^2)$
by the eigenvalues of the matrix
\begin{align*}
\sum_{k=0}^2 \varepsilon^k A_k=
\begin{pmatrix}
  g_n(\theta) + \varepsilon^2 a_{n,n} & \varepsilon^2 a_{-n,n} \\
   \varepsilon^2 \overline{a}_{-n,n} & g_{-n}(\theta) + \varepsilon^2 a_{-n,-n}\\
\end{pmatrix}
\end{align*}
where for $k=  \pm n, j = \pm n $, the matrix coefficients of $A_2$ are
\[
   a_{j,k}(\theta) = \left(M_2\right)_{jk} + \frac{1}{2} 
    \sum_{\ell \neq  \pm n } \left(M_1\right)_{j \ell}\left(T_1\right)_{\ell k} 
   - \left(T_1\right)_{j \ell}\left(M_1\right)_{\ell k} ~.
\]
From \eqref{pn-exp-pn}
\begin{align*}
  \left(M_1\right)_{jk}(\theta) & = -s_j(\theta)s_{k}(\theta) \beta_{j-k} \\
   \left(M_2\right)_{j k}(\theta) & = -\sum_{p \in \Z }
      s_j(\theta) \beta_{j-k-p} g_{k + p}(\theta) \beta_{p} s_k(\theta) 
   \\ 
    \left(T_1\right)_{(\pm n) \ell}(\theta) & = 
     \frac{\left(M_1\right)_{\pm n, \ell}(\theta)}{g_{\ell}(\theta) - g_{\pm n} (\theta)} 
     = \frac{s_{\pm n}(\theta)s_{\ell}(\theta) \beta_{\pm n -\ell}}{g_{\pm n }(\theta)
       - g_{\ell}(\theta)} ~. 
\end{align*}
Explicitly,
\begin{align*}
   a_{n,n} & = - s_n^2(\theta)  \Big(\sum_k |\beta_k|^2 g_{k+n}(\theta) 
      - \sum_{l\ne \pm n} \frac{ s_l^2(\theta)
        |\beta_{n-l}|^2}{g_n(\theta)- g_l(\theta)} \Big) ~, \nonumber \\
   a_{-n,-n} & = -s_{-n}^2(\theta) \Big( \sum_k |\beta_k|^2
   g_{k-n}(\theta) - \sum_{l\ne \pm n} 
      \frac{ s_l^2(\theta) |\beta_{n+l}|^2}{g_{-n}(\theta) -
        g_l(\theta)}  \Big) ~, \nonumber \\
   a_{n,-n} & =  -s_{-n}(\theta) s_n(\theta)  \sum_{k}  g_{-k}(\theta)\beta_{n+k} 
   {\beta}_{n-k} \nonumber \\
   & \quad + \frac{1}{2} s_n(\theta) s_{-n}(\theta) 
     \Big( \sum_{l\ne \pm n} {\beta}_{n-l}  \beta_{l+n} s_l^2(\theta) 
     (\frac{1}{g_{n}(\theta) - g_l(\theta)} 
    + \frac{1}{g_{-n}(\theta) - g_l(\theta)} )\Big) ~,  \nonumber \\
   a_{-n,n} & = \overline{a_{n,-n}} ~. 
\end{align*}

In a neighborhood of $\theta=0$, one has 
$g_n(\theta) - g_{-n}(\theta) = \BigOh{\theta}$ 
and $s_n(\theta) + s_{-n}(\theta) = \BigOh{\theta}$, 
so that the terms on the diagonal of $A(\theta)$ satisfy 
$A_{n,n}(\theta) - A_{-n,-n}(\theta) = \BigOh{\theta}$.
Thus, at $\theta=0$, using that $A_1=0$, we find that 
\begin{align*}
   \Lambda_{2n}(0)-\Lambda_{2n-1}(0) & = 2 \varepsilon^2 |a_{n,-n} | + \BigOh{\varepsilon^3}
\end{align*}
where
\begin{align} \label{an--n}
  a_{n,-n} (0)& 
    = -s_n(0)^2 \sum_{l\ne \pm n } \beta_{n-l}\beta_{n+l} ( -g_l(0) +
    \frac{s_l^2(0) }{g_n(0)-g_l(0)}) 
    \\ \nonumber
    &=-s_n(0)^2 \sum_{l\ne \pm n } \beta_{n-l}\beta_{n+l}
      \frac{l^2 -g_n(0)g_l(0)}{g_n(0)-g_l(0)} ~.
\end{align}
As long as the function $\beta(x)$ describing the bathymetry has
nonzero Fourier coefficients $\beta_l$ such that for some 
$l \not= \pm n$ both $\beta_{l\pm n} \not= 0$, then this quantity has
the possibility to be nonvanishing. In any case it is generically
nonzero. In this situation, the above expression gives a description of
the asymptotic size of the $n^{th}$ gap opening at order
$\mathcal{O}(\varepsilon^2)$.

\begin{proposition}\label{Prop:GapOpening}
The matrix coefficients of $A(\theta, \varepsilon)$ defined in \eqref{matrix-A} satisfy 

(i)  $A_{n,n}(0,\varepsilon) = A_{-n,-n} (0,\varepsilon) $ and it is real.

(ii) $A_{n,-n}(0,\varepsilon) = \overline{A}_{-n,n} (0,\varepsilon) $
due to the Hermitian character. A necessary condition for a gap to
open under perturbation in $\varepsilon$ is that $A_{n,-n} \ne 0$. 

(iii) Let $\mathcal{K} = \{ k_j \in \Z, \  \beta_{k_j}\ne0\} :=
\mathcal{K}(\beta)$, the set of Fourier indices corresponding to
nonzero  coefficients of $\beta(x)$. Suppose that, for all $p\le q$,
and for every sum of integers $\sum_{j=1}^p k_j =n$,  one of the $k_j$
does not belong to $\mathcal{K}$.  Then there is an upper bound of the
$n^{th}$ gap opening  
\[
   |\Lambda_n^- -\Lambda_{n-1}^+ | \le C \varepsilon^{q+1} ~.
\]

(iv) $A(0,\varepsilon) $ is analytic in $\varepsilon$. Thus, it is
determined by its Taylor series in $\varepsilon$. If 
$ \frac{\partial^q}{\partial \varepsilon^q} A_{n,-n} (0,0) =0 $ for
all $q$, then $A_{n,-n}(0,\varepsilon) =0$ for all $\varepsilon$. In
this case, the eigenvalues of $A(0,\varepsilon)$ are double, that is
the $n^{th}$ gap remains closed. 
\end{proposition}

\begin{proof}
We study the case of an even numbered gap. The odd numbered gaps are similar. 
$\frac{\partial^p}{\partial \varepsilon^p} A_{n,-n}(0,0)$ is the
$p^{th}$ Taylor coefficient of the matrix coefficient of $A$. It
satisfies the conclusion of Lemma \ref{lemma4-2}. It is a polynomial
of order $p$ in the Fourier coefficients $\beta_k$ : 
\[
   (A_p)_{n,-n} = \sum a_p(k_1,....,k_p) \beta_{k_1}....\beta_{k_p}
\]
where $k_1+ \cdots+k_p= 2n$. Under the hypothesis (iii), if $p\le q$, then $(A_p)_{n,-n}=0$.

The proof of (iv) follows from Theorem \ref{theor-IFT} and the
analyticity properties of $M$ with respect to $b$. 
\end{proof}

\subsection{Examples}
When  $b(x) =\varepsilon \cos (x)$, the first gap occurs for 
$\theta = \pm 1/2$ and is of order $ \BigOh{\varepsilon}$ indeed by \eqref{app-diff}. 
\[
    \Lambda_1^-  -\Lambda_0^+ = g_0(\frac{1}{2})  2 | s_0(\frac{1}{2})|^2 |\beta_1|
    \varepsilon  = \frac{1}{4} \sech^2(\frac{h}{2}) \varepsilon   
     ~.  
\]
We are also able to calculate analytically the deviation of the centre.
 We find that 
\begin{equation}
 \frac{1}{2} (\Lambda_1^-   +\Lambda_0^+) = -\ \varepsilon^2
 \frac{s_0^2(\frac{1}{2})  ( g_0^2(\frac{1}{2}) - \frac{9}{4})  }{4
   (g_0(\frac{1}{2}) -g_1(\frac{1}{2}))} 
 +  \BigOh{\varepsilon^3}.
 \end{equation}
 It is straightforward to check that this quantity is negative,
 showing that the centre of the gap is strictly decreasing with
 $\varepsilon$. 
Hence for increasing $\varepsilon$ the gap centre is transposed, or
downshifted, to lower frequency. This is an analytical verification of
Figure 2b of reference \cite{YH12}. 

The second gap occurs at $\theta=0$. We find that for $n=1$ the coefficient
$a_{n,-n}$ of \eqref{an--n} vanishes. We have remarked above that this
is in contrast with the case of Matthieu's equation.
 Calculating  the expansion further and 
 using \eqref{Tq-G} 
 we find that 
\[ 
   A_{3} = P_n \Big( \frac12[M_{2},T_{1}]+\frac12[M_{1},T_{2}]+M_{3}\Big) P_{n}
\]
and 
that $A_3(0)$ is a multiple of the identity, and in particular the
off-diagonal terms satisfy $(A_3)_{1,-1} = 0$, thus not contributing
to the formation of a spectral gap.  Continuing the perturbation calculation,
\begin{align*}
 A_4 = 
 P_n \Big(& M_4-  [T_1,M_3] - [T_2,M_2] +\frac12 [T_{1},[T_{1},M_{2}]]
- [T_3,M_1] \\
& - \frac16 [T_1,[T_1,[T_1,M_1] ] ] + \frac12 [T_{1},[T_{2},M_{1}]] +
\frac12  [T_{2},[T_{1},M_{1}]] \\ 
&+\frac1{24}[T_{1}, [T_1,[T_1,[T_1,G_{\theta}] ] ] ] \\
&+ \frac12 [T_{2},[T_{2},G_{\theta}]] + \frac12
         [T_{1},[T_{3},G_{\theta}]]+ \frac12
         [T_{3},[T_{1},G_{\theta}]]  \Big) P_n
\end{align*}
which can be simplified using \eqref{Tq-G} to
\begin{equation*}
   A_4 = 
   P_n \Big( M_4-\frac{1}{2} [T_1,M_3] -\frac{1}{2} [T_2,M_2]
    - \frac{1}{2} [T_3,M_1] +\frac{1}{24} [T_1,[T_1,[T_1,M_1] ] ] 
   \Big) P_n.
\end{equation*}
where for $\theta=0$,
\begin{align*}
&(M_4)_{1,-1} = -\frac{1}{3} s_1(0) s_{-1}(0) g_2(0) \beta_1^3
  \beta_{-1} = \frac{1}{48} s_1(0)^2 g_2(0) \nonumber \\ 
&([T_1,M_3])_{1,-1} = (T_1)_{1,2} (M_3)_{2,-1} - (M_3)_{1,-2} (T_1)_{-2,-1} \nonumber \\
& \qquad = -\frac{1}{48} s_2^2(0) s_1(0)^2\frac{1}{g_2(0) -g_1(0)} \nonumber \\
&([T_2,M_2])_{1,-1} = (T_2)_{1,-3} (M_2)_{-3,-1} - (M_2)_{1,3})T_2)_{3,-1} =0\nonumber \\
&([T_3,M_1])_{1,-1} = (T_3)_{1,-2} (M_1)_{-2,-1} - (M_1)_{1,2} (T_3)_{2,-1} \nonumber \\
& \qquad =\frac{1}{48} s_1(0)^2 s_2^2(0)  \frac{1}{g_2(0) -g_1(0)}  \nonumber\\
&([T_1,[T_1,[T_1,M_1]]])_{1,-1}=0 ~.
\end{align*}
Using this lengthy but straightforward calculation, the off-diagonal
entries $(A_4)_{1,-1} = (\overline{A_4})_{-1,1}$ of the matrix $A_4$,
are given by   
\[
   (A_4)_{1,-1} = \frac{1}{48} s_1^2(0) g_2(0) = \frac{1}{24} \sech^2(h) \tanh (2h)
\]
which are nonvanishing. We thus have 
\begin{equation}\label{Eqn:SecondGap-Matthieu}
   \Lambda_2^- - \Lambda_1^+ = 2 \varepsilon^4 |(A_4)_{1,-1}| ~,
\end{equation}
establishing a gap opening of order $\BigOh{\varepsilon^4}$.
In general the $n^{th}$ gap satisfies 
\[
  \Lambda_n^-  -\Lambda_{n-1}^+ \le C(n)  \varepsilon^n ~,
\]
which follows from Proposition~\ref{Prop:GapOpening}. 

On the other hand, when $\beta(x) = \cos(x) +  \cos(3x)$, the second gap
opens at order $O(\varepsilon^2)$, indeed in expression
\eqref{an--n}, there is a non-zero term in the sum, which corresponds
to $l=\pm 2 :$ 
\begin{align*}
    \Lambda_2^-  -\Lambda_1^+ &=  \varepsilon^2 s_1^2(0)
    \Big| -g_2(0) +\frac{s_2^2(0)}{g_1(0)-g_2(0)} \Big| \nonumber \\
    &= \varepsilon^2 \sech^2 (h) \frac{4-2 \tanh(h) \tanh(2h)}{\tanh(h) - 2 \tanh(2h)}.
\end{align*}
In general, unlike the case of the Matthieu operator, the $n^{th}$ gap
does not necessarily open at order $\varepsilon^n$ due to the
combinatorics of the perturbation analysis of the spectrum of the
Dirichlet -- Neumann operator.

When $\beta(x)= \cos(2x)$, the upper bound on the $n^{th}$ gap for odd $n$ in
criterion (iii) of Proposition \ref{Prop:GapOpening} is satisfied for
all $q$. Thus (iv) applies and the odd gaps never open. 
 

\section{Proof of Proposition \ref{prop2}}


The goal of this section is to prove Proposition \ref{prop2} that
shows the smoothness property of the operator $M$. 
Returning to the definition of $M= e^{-i\theta x} DL[b] e^{i\theta x}$
where $DL$ is given in \eqref{operatorDL}-\eqref{operatorsAB} 
we write
\[ 
   M = e^{-i\theta x} D B[b] e^{i\theta x}e^{-i\theta x}  A[b]
   e^{i\theta x} ~.
\]
For $f \in L^2(\T^1)$, the action of the operator $A[b]$ given in
\eqref{operatorsAB} on $e^{i\theta x} f$ is 
\begin{align*}
A[b] ( e^{i\theta x} f) =\int_{\R} e^{ikx} \sinh(b(x) k) \sech(hk)  \hat{f}(k-\theta) dk  ~. 
\end{align*}
Using the periodicity of $f(x)$ and $b(x)$,  we can write integral kernels as a sum:
\begin{align*}
   A[b] (e^{i\theta x} f)&= \sum_k \Big(  e^{i(k+\theta)x} \sinh(b(x)
     (k+\theta)) \sech(h(k+\theta))\hat{f}_k \Big) \\
   & = \frac{1}{2\pi} \int_0^{2\pi} f(x') e^{i\theta  x'}  \sum_k
       \sinh(b(x) (k+\theta)) \sech(h(k+\theta)) e^{i(k+\theta)
         (x-x')}  dx' \\ 
   & = \frac{1}{2\pi} \int_0^{2\pi} f(x') e^{i\theta x'}  K(x,x',\theta) dx'
\end{align*}
where the kernel 
$K(x,x',\theta) = \sum_k \sinh(b(x) (k+\theta)) \sech(h(k+\theta)) e^{i(k+\theta) (x-x')}$
has the property that 
\[
  | K(x,x',k)| \le C \sum_k e^{-|k+\theta| (h-|b(x)|)} ~, 
\]
with the rhs being  a convergent series as soon as $|b(x)|$ remains
always strictly smaller than the depth $h$. The operator $A[b]$ satisfies
the estimate 
\begin{equation}
   \|e^{-i\theta x} A[b] e^{i\theta x} f\|_{L^2(\T^1)} \le C(|b|_{C^0}) \|f\|_{L^2(\T^1)}
\end{equation}
as well as a stronger form of it
\begin{equation}
   \|e^{-i\theta x} A[b] e^{i\theta x} f\|_{L^2(\T^1)} \le C(|b|_{C^0}) \|f\|_{H^{-s}(\T^1)} ~.
\end{equation}

We now examine the operator $DB[b] $ acting on $\theta$-periodic functions. 
\begin{align*}
\partial_x B[b] (e^{i\theta x} f) &= -\frac{1}{\pi}  \int_\R \frac{\partial_{x'} b(x')}
{ (x-x')^2 + (b(x')-h)^2} e^{i\theta x'}  f(x') dx' \cr
&+ \frac{2}{\pi}
\int_\R \frac{ \partial_{x'} b(x')  (x'-x)^2}
{ [(x-x')^2 + (b(x')-h)^2]^2} e^{i\theta x'}  f(x') dx' \cr
& -   \frac{1}{\pi} \int_\R \frac{x-x'}{(x-x')^2 + (h-b(x'))^2)}
\widetilde{G}[-h+b] e^{i\theta x'}  f(x') dx'. 
\end{align*}
It has three terms that we denote $I_1, I_2,I_3$ respectively.

Using the periodicity of $f$ and $b$, we can replace integrals over
$\R$ by sums of integrals over $(0,2\pi)$:  
\begin{equation}
   I_1 = -\frac{1}{\pi} \int_0^{2\pi} f(x') \partial_{x'} b(x')
   e^{i\theta x'} K_1( x,x',\theta) dx' ,
\end{equation}
where 
\[
   K_1(x,x', \theta) = \sum_n \frac{e^{2\pi n i\theta}}{(x-x'-2\pi n)^2
     + (b(x') - h)^2} 
\]
satisfies, using that $|b(x)-h| > \alpha$, 
\[
   |K_1(x,x', \theta)|\le \sum_n \frac{1}{ (x-x'-2\pi n)^2 +\alpha^2} ~.
\]
The kernel $K$ is bounded by a convergent series, thus it is bounded
and it is also in $L^1$. Thus 
\begin{equation}
   \|e^{-i\theta x} I_1\|_{L^2(\T^1)}  \le  C(|b|_{C^1}) \| f \|_{L^2(\T^1)} ~.
\end{equation} 
Also, if we take derivatives of $B[b]$, they will not act on $b$,  and the
resulting terms in the integrals will decrease faster. Thus we also
have the smoothing property
\begin{equation}
   \|e^{-i\theta x} I_1\|_{H^p(\T^1)}  \le  C(|b|_{C^1}) \| f \|_{L^2(\T^1)} 
\end{equation} 
for all $p>0$.
The estimate for term $I_2$ is similar. For term $I_3$, we use that
the Dirichlet -- Neumann $\widetilde{G}[-h+b]= D \widetilde{H}[-h+b]$
where  $\widetilde{H}$ is the Hilbert transform associated to the
spatial domain $ -h+b(x) < y< 0$, described in  \cite{CSS92}. By
integration by parts, we have  
\begin{align*}
   I_3 = \frac{1}{\pi} \int_\R  \partial_{x'}
   \Big(\frac{x-x'}{(x-x')^2 + (h-b(x'))^2 }\Big) i
   \widetilde{H}[-h+b]  (e^{i\theta x'}  f(x')) dx' ~. 
\end{align*} 
The estimates are now similar to those of terms $I_1$ or $I_2$, using
the fact that $\widetilde{H}$ is a  bounded operator from $L^2$ to
$L^2$. 

\bigskip
\noindent
{\bf Acknowledgements:} We would like to thank Philippe Guyenne for
useful discussions which helped us in preparing this manuscript.

\end{document}